\documentclass[11pt]{amsart}
\usepackage{amssymb}
\usepackage{amscd}
\usepackage{amsmath}
\usepackage{amsmath}
\usepackage{enumitem}
\usepackage[all]{xy}
\usepackage{mathrsfs}
\usepackage{stmaryrd}
\usepackage{tikz-cd}
 \usepackage[backend=bibtex, style=alphabetic, sorting=nyt, maxbibnames=9]{biblatex}
		\renewbibmacro{in:}{%
  \ifentrytype{article}{}{\printtext{\bibstring{in}\intitlepunct}}}
  \DeclareFieldFormat{journaltitle}{#1\isdot}
  \DeclareFieldFormat[article, inproceedings, unpublished]{title}{\textit{#1}}

\usepackage{color}
\usepackage[colorlinks, pdfpagelabels,
	pdfstartview = FitH,
	bookmarksopen = true,
	bookmarksnumbered = true,
	linkcolor = black,
	plainpages = false,
	hypertexnames = false,
	citecolor = black, urlcolor=black]{hyperref}

\newcommand{\DR}{\mathrm{R}}
\newcommand{\DL}{\mathbb{L}}

\usepackage{dsfont}
\renewcommand{\mathbb}{\mathds}

\theoremstyle{plain}
\newtheorem{theorem}{Theorem}[section]
\newtheorem{lemma}[theorem]{Lemma}

\newtheorem{proposition}[theorem]{Proposition}

\newtheorem{conjecture}[theorem]{Conjecture}
\theoremstyle{definition}

\newtheorem{example}[theorem]{Example}

\newtheorem{hypothesis}[theorem]{Hypothesis}
\theoremstyle{definition}
\newtheorem{definition}[theorem]{Definition}
\newtheorem{remark}[theorem]{Remark}

\newtheorem*{acknowledgments}{Acknowledgements}

\font\russ=wncyr10  1
\def\sha{\hbox{\russ\char88}}

\DeclareMathOperator{\Cl}{Cl}

\DeclareMathOperator{\Ext}{Ext}
\DeclareMathOperator{\Gal}{Gal}

\DeclareMathOperator{\Hom}{Hom}
\DeclareMathOperator{\End}{End}

\DeclareMathOperator{\Spec}{Spec}

\DeclareMathOperator{\res}{res}
\DeclareMathOperator{\bigO}{\mathcal{O}}
\DeclareMathOperator{\im}{im}
\DeclareMathOperator{\coker}{coker}

\DeclareMathOperator{\Fitt}{Fitt}

\newcommand{\CC}{\mathbb{C}}

\newcommand{\EE}{\mathrm{E}}

\newcommand{\GG}{\mathbb{G}}

\newcommand{\QQ}{\mathbb{Q}}
\newcommand{\Q}{\mathbb{Q}}

\newcommand{\ZZ}{\mathbb{Z}}
\newcommand{\Z}{\mathbb{Z}}

\newcommand{\cG}{\mathcal{G}}

\newcommand{\cO}{\mathcal{O}}

\newcommand{\cR}{\mathcal{R}}

\newcommand{\fz}{\mathfrak{z}}

\renewcommand{\:}{\colon}

\renewcommand{\det}{\mathrm{det}}

\newcommand{\cyc}{\mathrm{cyc}}
\newcommand{\dR}{\mathrm{dR}}

\newcommand{\p}{\mathfrak{p}}

\newcommand{\Ann}{\mathrm{Ann}}

\newcommand{\B}{\mathrm{B}}

 \newcommand{\bidual}{\bigcap\nolimits}
\newcommand{\exprod}{\bigwedge\nolimits}

\newcommand{\rgamma}{\mathrm{R}\Gamma}
\newcommand{\rhom}{\mathrm{R}\Hom}

\mathchardef\ordinarycolon\mathcode`\:
\mathcode`\:=\string"8000
\begingroup \catcode`\:=\active
  \gdef:{\mathrel{\mathop\ordinarycolon}}
\endgroup

\makeatletter
\def\l@subsection{\@tocline{2}{0pt}{2.5pc}{5pc}{}}
\newcommand\@dotsep{4.5}
\def\@tocline#1#2#3#4#5#6#7{\relax
  \ifnum #1>\c@tocdepth 
  \else
    \par \addpenalty\@secpenalty\addvspace{#2}%
    \begingroup \hyphenpenalty\@M
    \@ifempty{#4}{%
      \@tempdima\csname r@tocindent\number#1\endcsname\relax
    }{%
      \@tempdima#4\relax
    }%
    \parindent\z@ \leftskip#3\relax \advance\leftskip\@tempdima\relax
    \rightskip\@pnumwidth plus1em \parfillskip-\@pnumwidth
    #5\leavevmode\hskip-\@tempdima{#6}\nobreak
    \leaders\hbox{$\m@th\mkern \@dotsep mu\hbox{.}\mkern \@dotsep mu$}\hfill
    \nobreak
    \hbox to\@pnumwidth{\@tocpagenum{#7}}\par
    \nobreak
    \endgroup
  \fi}
\makeatother
\AtBeginDocument{%
\makeatletter
\expandafter\renewcommand\csname r@tocindent0\endcsname{0pt}
\makeatother
}
\def\l@subsection{\@tocline{2}{0pt}{2.5pc}{5pc}{}}

\begin{document}

\title[]{On $p$-adic families of special elements \\
for rank-one motives}

\author{Dominik Bullach, David Burns and Takamichi Sano}

\begin{abstract}
We conjecture that special elements associated with rank-one motives are obtained $p$-adically from Rubin-Stark elements by means of a precise \textit{higher-rank Soul\'e twist} construction. We show this conjecture incorporates a variety of known results and existing predictions and also gives rise to a concrete strategy for proving the equivariant Tamagawa Number Conjecture for rank-one motives. We then use this approach to obtain new evidence in support of the equivariant Tamagawa Number Conjecture in the setting of CM abelian varieties. 
\end{abstract}

\address{King's College London,
Department of Mathematics,
London WC2R 2LS,
U.K.}
\email{dominik.bullach@kcl.ac.uk}

\address{King's College London,
Department of Mathematics,
London WC2R 2LS,
U.K.}
\email{david.burns@kcl.ac.uk}

\address{Osaka  City University,
Department of Mathematics,
3-3-138 Sugimoto\\Sumiyoshi-ku\\Osaka\\558-8585,
Japan}
\email{sano@sci.osaka-cu.ac.jp}

\maketitle
\tableofcontents


\section{Introduction} \label{Intro}

Let $M$ be a (pure) motive defined over a number field $K$ and endowed with an action of a finite dimensional  commutative semisimple $\QQ$-algebra $R$. For a prime $p$ set $R_p := \QQ_p\otimes_\QQ R$ and write $V^\ast(1)$ for the Kummer dual of the $p$-adic realization $V$ of $M$.

In \cite{bss2} Sakamoto and the second and third of the present authors defined a canonical
\textit{Bloch-Kato~element} $\eta_M$ that lies in $\CC_p\otimes_{\QQ_p}{\bigwedge}^r_{R_p}H^1(K,V^\ast(1))$ for an appropriate non-negative integer $r$ (that depends on $M$ and $R$). These elements simultaneously generalize several well-known families of special elements, including the \textit{Rubin-Stark elements} defined by Rubin \cite{rubinstark} for the multiplicative group and the \textit{zeta elements} constructed by Kato \cite{kato} for elliptic curves over $\QQ$.\medskip \\
%
%
In this note we restrict attention to motives that are of rank one in the sense of Deligne \cite{deligneconj}, and predict a precise family of relations  between $\eta_{M}$ and $\eta_{M'}$ for differing such motives $M$ and $M'$.

The source of these relations is a natural higher-rank generalization of the notion of  \textit{Soul\'e twists} (in the terminology used by Loeffler and Zerbes, see \cite[\S 1.4c]{LZ}) that uses
 the theory of exterior power biduals to overcome technical problems that arise when dealing with torsion coefficients. For a rank-one motive we shall then define the \textit{Soul\'e-Stark element}  to be an appropriate \textit{higher-rank Soul\'e twist} of Rubin-Stark elements (see Definition \ref{def soule stark}) and predict that this element coincides with the corresponding Bloch-Kato element. This prediction will be referred to as the \textit{Soul\'e-Stark Conjecture} (see Conjecture \ref{conj}) and entails precise, and explicit, relations between Bloch-Kato elements and Rubin-Stark elements.

We will see, for example, that for motives arising from the multiplicative group $\mathbb{G}_m$ the Soul\'e-Stark Conjecture predicts explicit relations between the \textit{generalized Stark elements} introduced in \cite{bks2-2} that are essentially different from the relations  that are investigated in \cite{sbA2} via a study of `functional equations for Euler systems', whilst for motives arising from CM abelian varieties the Soul\'e-Stark Conjecture extends, and refines, the `explicit reciprocity conjecture' that is formulated by B\"uy\"ukboduk and Lei in \cite{BL1}.

However, aside from any intrinsic interest that the Soul\'e-Stark Conjecture  might have in particular cases, we show that its validity would also establish a precise connection between the main conjecture of higher-rank equivariant Iwasawa theory for $\mathbb{G}_m$, as formulated explicitly by Kurihara et al.\@ in \cite[Conj.\@ 3.1]{bks2}, and the equivariant Tamagawa Number Conjecture for a general rank-one motive. This result is stated precisely as Theorem \ref{descent} and establishes a natural refinement and generalization to the setting of rank-one motives with coefficients of the main strategy that was used by Huber and Kings \cite{HK} to prove the Tamagawa Number Conjecture for Tate motives over abelian extensions of $\QQ$. We remark that this approach is thus essentially different from that used by Greither and the second author \cite{BG} and, more generally, by Kurihara et al.\@ \cite{bks2} which relies on a study of the Mazur-Rubin-Sano Conjecture (see Remark \ref{remark descent}).\newpage
For this approach to be of any practical use, one must of course understand both the main conjecture of higher-rank equivariant Iwasawa theory for $\mathbb{G}_m$ and the relevant special cases of the Soul\'e-Stark Conjecture. 

The first of these issues is addressed in \S\ref{mc Gm section} where we use the recently developed `higher-rank Kolyvagin-derivative' techniques of Sakamoto et al.\@ \cite{bss} to obtain concrete new evidence in support of \cite[Conj.\@ 3.1]{bks2}. Our main result in this regard is stated as Theorem \ref{classical-imc-result} and refines earlier results of B\"uy\"ukboduk  in \cite{buyukboduk-1} and \cite{buyukboduk-2} and of B\"uy\"ukboduk and Lei in \cite{BL1}. 

Then, in \S\ref{ss section}, we establish explicit relations between important cases of Conjecture \ref{conj}  and results and conjectures already existing in the literature, thereby deriving concrete evidence in support of the Soul\'e-Stark Conjecture in these cases. 

In \S\ref{tr cm section} we show firstly that the Soul\'e-Stark Conjecture incorporates a wide variety of known facts and existing predictions relating to invariants of $\mathbb{G}_m$, ranging from the interpolation properties of Deligne-Ribet $p$-adic $L$-functions to the results of Beilinson and Huber-Wildeshaus
on the cyclotomic elements of Deligne-Soul\'e, the $p$-adic Beilinson conjecture of Besser-Buckingham-de\,Jeu-Roblot \cite{BBJR} and the explicit reciprocity law for Rubin-Stark elements conjectured by Solomon \cite{solomon}. 

In \S\ref{BL-section}, we shall then show that, in the setting of motives arising from CM abelian varieties, the Soul\'e-Stark Conjecture both extends and refines the explicit reciprocity conjecture studied in \cite{BL1} and \cite{BL2}. Upon combining this observation with the result of Theorem \ref{classical-imc-result} we are then able to derive concrete new evidence in support of the equivariant Tamagawa Number Conjecture in this case (see Theorem \ref{final cor}) and thereby also refine some of the main results of \cite{BL1}. 

Finally, we note that the approach developed here will also allow us to clarify other aspects of the results and conjectures in  \cite{BL1} and \cite{BL2} (for details of which see \S\ref{Perrin-Riou-Stark-section} and \S\ref{exp rec BL}).

\begin{acknowledgments}
Some of the material presented in \S\ref{bk stark section} and \S\ref{cong conj section} below is an updated version of results in the (unpublished) arXiv-version of the article \cite{bks2-2} of Masato Kurihara and the second and third authors. We are very grateful to Kurihara for permission to include that material in this article and, more generally, for many interesting discussions, encouragement and advice. We are also grateful to Kazim B\"uy\"ukboduk for helpful correspondence concerning results in the articles \cite{BL1} and \cite{BL2}.\\ 
The first author wishes to acknowledge the financial support of the Engineering  and  Physical  Sciences  Research  Council [EP/L015234/1],  the  EPSRC  Centre  for  Doctoral  Training  in  Geometry  and  Number  Theory  (The  London School of Geometry and Number Theory), University College London and King's College London.
\end{acknowledgments}

\section{Bloch-Kato and Stark elements}\label{bk stark section}

In this section, we quickly review relevant aspects of the theory of Bloch-Kato elements from \cite{bss2} and also recall some basic facts concerning \'etale cohomology that will be useful in the sequel.

\subsection{Bloch-Kato elements}\label{review bk}

The motive $M$ discussed in the Introduction has the following realizations.
\begin{itemize}
\item For each prime number $p$, the $p$-adic \'etale realization $V_p(M)$: a finitely generated $R_p$-module endowed with a continuous $R_p$-linear action of $G_K:=\Gal(\overline \QQ/K)$.
\item For each embedding $\sigma \: K \hookrightarrow \CC$, the Hodge $\sigma$-realization $H_\sigma(M)$: a finitely generated $R$-module. 
\item The Betti realization $H_\B(M):=\bigoplus_{\sigma: K \hookrightarrow \CC} H_\sigma(M)$: a finitely generated $R$-module  endowed with an action of complex conjugation.
\end{itemize}

\begin{definition}
The $R$-rank $r(M/K,R)$ of $H_\B(M)^+$ is called the {\it basic rank} of the motive $(M/K,R)$. (This rank is, in general, a function $\Spec R \to \ZZ$, but we shall only consider examples in which it is constant.)
\end{definition}

We write $S_\infty(K)$ for the set of infinite places of $K$ and $S_\CC(K)$ for the subset of $S_\infty(K)$ of complex places.

\begin{example}\label{exrank}\

\begin{itemize}
\item[(i)] For any number field $K$ and integer $j$, we have
\[
r(h^0(K)(j)/K,\QQ)= \begin{cases}
 |S_\infty(K)| & \text{if $j$ is even},\\
|S_\CC(K)| & \text{if $j$ is odd}.
\end{cases}
\]
\item[(ii)] Let $L/K$ be a finite abelian extension with Galois group $G$. Assume $K$ is totally real and $L$ is CM and write $c$ for the complex conjugation in $G$. For an integer $j$ set
\[e_j^\pm :=\frac{1\pm (-1)^j c}{2} \in \QQ[G].\]
Then one has
\[r(h^0(L)(j)/K, \QQ[G]\varepsilon)= \begin{cases} [K:\QQ] &\text{ if $\varepsilon = e_j^+$,}\\
0 &\text{ if $\varepsilon = e_j^-$}.\end{cases}\]
\end{itemize}
\end{example}

Fix an odd prime number $p$, a Gorenstein $\ZZ$-order $\cR$ in $R$ and a $G_K$-stable lattice $T:=T_p(M)$ of $V:=V_p(M)$, which is free as an $\cR_p:=\ZZ_p\otimes_\ZZ \cR$-module.

Now we assume the following.

\begin{hypothesis}\label{hyp}
The $\cR_p$-module $Y_K(T):=\bigoplus_{v \in S_\infty(K)} H^0(K_v, T)$ is free.
\end{hypothesis}

\begin{remark}
Since the $R_p$-modules $\QQ_p \otimes_{\ZZ_p} Y_K(T)$ and $ \QQ_p \otimes_\QQ H_\B(M)^+$ are isomorphic, the rank of $Y_K(T)$ is equal to $r(M/K,R)$.
\end{remark}

We now set $r:=r(M/K,R)$ and fix a finite set $S $ of places of $K $ such that
\[S_\infty(K)\cup S_p(K) \cup S_{\rm ram}(T) \subset S,\]
where $S_p(K)$ denotes the set of $p$-adic places of $K$ and
$S_{\rm ram}(T)$ the set of places of $K$ at which $T$ is ramified.

Then, for each ordered $\cR_p$-basis $ b=\{b_1,\ldots,b_r\}$ of $Y_K(T)$, one can use the leading term of the $L$-function of $M$ at $s=0$ to define a canonical \textit{Bloch-Kato element}
\[
\eta_S^b (T) \in \CC_p \otimes_{\ZZ_p} {\bigwedge}_{\cR_p}^r H^1(\cO_{K,S},T^\ast(1)).
\]
(See \cite[Def.\@ 4.10]{bss2} for the precise definition.) Here $T^\ast(1):=\Hom_{\ZZ_p}(T,\ZZ_p(1))$ is the Kummer dual of $T$.

In order to study the integrality properties of Bloch-Kato elements we first recall that if $A$ is a commutative noetherian ring, then  for any non-negative integer $a$ the \textit{$a$-th exterior power bidual} of a finitely generated $A$-module $X$ is defined by
\[ {{\bigcap}}_A^a X:=\Hom_A\left({{\bigwedge}}_A^a \Hom_A(X,A), A\right).\]

We further recall that the canonical homomorphism of $A$-modules
\[
\xi_X^a \: {{\bigwedge}}_A^a X \to {{\bigcap}}_A^a X, \quad \ x \mapsto (\Phi \mapsto \Phi(x)),
\]
is bijective if $X$ is  projective or if both $a = 1$ and $X$ is reflexive, but is in general neither injective nor surjective.

In addition, if  $A = \mathcal{R}_p$, then \cite[Prop. A.7]{sbA} implies $\xi_X^a$ induces an isomorphism
\begin{equation}\label{rl=bidual rem}
\Big \{ x \in \QQ_p \otimes_{\ZZ_p}{\bigwedge}_{\mathcal{R}_p}^a  X \ \Big | \ \Phi(x) \in \mathcal{R}_p \text{ for all $\Phi \in {\bigwedge}_{\mathcal{R}_p}^a\Hom_{\cR_p}(X,\cR_p)$} \Big \} \stackrel{\sim}{\longrightarrow} {\bigcap}^a_{\mathcal{R}_p} X.
\end{equation}
In the sequel we use this map to regard   ${\bigcap}_{\mathcal{R}_p}^a X$ as a submodule of $\QQ_p \otimes_{\ZZ_p}{\bigwedge}_{\mathcal{R}_p}^a  X$.\medskip \\
We now consider the following hypothesis.

\begin{hypothesis}\label{hyptorsionfree}\
\begin{itemize}
\item[(i)] $H^0(K,T^\ast(1))=0$.
\item[(ii)] $H^1(\cO_{K,S},T^\ast(1))$ is $\ZZ_p$-free.
\end{itemize}
\end{hypothesis}

\begin{conjecture}
[Integrality Conjecture, {\cite[Conj.\@ 4.15]{bss2}}]\label{ic}
If Hypotheses \ref{hyp} and \ref{hyptorsionfree} are both valid, then one has
$$\eta_S^b(T) \in {\bigcap}_{\cR_p}^r H^1(\cO_{K,S},T^\ast(1)).$$
\end{conjecture}

\subsection{Stark elements}\label{review stark}

To discuss an important special case of the above conjecture we fix a finite abelian extension of number fields $L/K$ with Galois group $G$.

For each  $\ZZ_p \llbracket G_K \rrbracket$-module $X$ we set
\[ X_{L/K} := {\rm Ind}_{G_L}^{G_K}(X),\]
regarded as endowed with the natural action of $G\times G_K$.

We consider the Tate motive $(h^0(L)/K,\QQ[G])$. As an order of $\QQ[G]$ we take $\ZZ[G]$. Then the $p$-adic \'etale realization of this motive is $\QQ_{p,L/K}$ and contains the natural lattice $\ZZ_{p,L/K}$.

Suppose that all places in $S_\infty(K)$ split completely in $L$. Then $r:=r(h^0(L)/K,\QQ[G])$ is equal to $| S_\infty(K)|$ and Hypothesis \ref{hyp} is satisfied. Choosing an ordered basis $b$ of $Y_K(\ZZ_{p,L/K})$ is equivalent to choosing a labeling $v_1,\ldots,v_r$ of the infinite places of $K$ and a place $w_i$ of $L$ lying above each $v_i$. If we write
$$\eta_{L/K,S}^b \in \CC_p \otimes_{\ZZ_p} {\bigwedge}_{\ZZ_p[G]}^r H^1(\cO_{L,S},\ZZ_p(1))$$
for the Rubin-Stark element defined by this choice (see \cite[\S2A]{bks2},  for example), then
\begin{equation}\label{rs}\eta_S^b(\ZZ_{p,L/K})=\eta_{L/K,S}^b.\end{equation}

More generally, for an integer $j$ and an idempotent $\varepsilon \in \ZZ_p[G]$, let
\[\eta_{L/K,S}^\varepsilon(j) \in \CC_p\otimes_{\ZZ_p} {\bigwedge}_{\ZZ_p[G]}^{r}H^1(\cO_{L,S},\ZZ_p(1-j))\]
be the generalized Stark element defined in \cite[Def. 2.9]{bks2-2}, where $r:=r_j^\varepsilon$ is as in \cite[\S 2.1]{bks2-2}. Then we have
\[\eta_S^b(\varepsilon\cdot\ZZ_p(j)_{L/K})=\eta_{L/K,S}^\varepsilon(j),\]
where we take $b$ to be the dual of the basis in \cite[Lem.\@ 2.1]{bks2-2}.

\begin{remark}\label{sigma} Assume that $T=\varepsilon\cdot\ZZ_p(j)_{L/K}$ for some integer $j$. Then Hypothesis \ref{hyptorsionfree}\,(i) is satisfied unless $j=1$, in which case it is satisfied if and only if $\varepsilon$ belongs to the augmentation ideal of $\ZZ_p[G]$. Hypothesis \ref{hyptorsionfree}\,(ii) need not be satisfied in this case (even if $j \neq 1$) but this issue is easily avoided by choosing an auxiliary set $\Sigma$ of places of $K$ and using a notion of $\Sigma$-modified cohomology as in \cite{bks2-2} (where $\Sigma$ is denoted by $T$). In particular, when we consider examples in the sequel for which $T=\varepsilon\cdot\ZZ_p(j)_{L/K}$ and Hypothesis \ref{hyptorsionfree}\,(ii) is not satisfied, it should be understood that a set $\Sigma$ is implicitly used. \end{remark}

\begin{remark} If $T$ is as in Remark \ref{sigma}, then Conjecture \ref{ic} with $\cR_p = \ZZ_p[G]\varepsilon $ is known to be valid in each of the following cases:
\begin{itemize}
\item[$\bullet$] $L$ is an abelian extension of $\QQ$ (this is due to Kurihara and the second and third authors \cite[Thm.\@ 4.1]{bks2-2});
\item[$\bullet$] $K$ is totally real, $L$ is CM, $j\le 0$, $\varepsilon$ is the idempotent $e_j^-$ in Example \ref{exrank}\,(ii) (this is due to Deligne and Ribet (cf.\@ \cite[Ex.\@ 3.10\,(i)]{bks2-2})).
    \end{itemize}
\end{remark}

\begin{remark}\label{RSremark} After taking account of the final comment in Remark \ref{sigma}, the equality (\ref{rs}) combines with the identification (\ref{rl=bidual rem}) to imply that if $\cR_p = \ZZ_p[G]$ and $T = \ZZ_{p,L/K}$, then Conjecture \ref{ic} coincides with the $p$-component of the \textit{Rubin-Stark Conjecture} for $L/K$, as formulated by Rubin in \cite{rubinstark}.
\end{remark}

\subsection{Galois cohomology}\label{gal coh section}

In this section we recall some basic facts about compactly supported \'etale cohomology that will be useful in the sequel. 

For a commutative noetherian ring $\mathcal{A}$ we write $D(\mathcal{A})$ for the derived category of $\mathcal{A}$-modules.  

Let $N$ be a continuous $\ZZ_p\llbracket G_K\rrbracket$-module that has a commuting action of $\mathcal{A}$. Then, if $S$ is any finite set of places of $K$ that contains $S_\infty(K)$, $S_p(K)$ and all places at which $N$ is ramified one can regard $N$ as an \'etale pro-sheaf of $\mathcal{A}$-modules on ${\rm Spec}(\mathcal{O}_{K,S})$ and hence define its compactly supported \'etale cohomology complex $\mathrm{R}\Gamma_c(\mathcal{O}_{K,S}, N)$ (as discussed, for example, in \cite[\S 1.4]{sbA}). This complex defines an object of $D(\mathcal{A})$ that is well-defined up to unique isomorphism and so its shifted linear dual    
\[ C_{\mathcal{A}}(N) := \rhom_{\mathcal{A}}(\rgamma_c(\mathcal{O}_{K,S}, N), \mathcal{A}[-3]) \]
is also an object of $D(\mathcal{A})$ that is well-defined up to unique isomorphism.

In the following result we record some useful facts about this construction.

\begin{lemma}\label{compact lemma} Let $\mathcal{A}$ be a Gorenstein $\Z_p$-order and $N$ a finitely generated free $\mathcal{A}$-module that has a commuting continuous action of $G_K$ that is unramified outside a finite set of places $S$ of $K$ that contains both $S_\infty(K)$ and $S_p(K)$. Then the following claims are valid. 
\begin{enumerate}[label=(\roman*)]
\item $C_{\mathcal{A}}(N)$ is a perfect complex of $\mathcal{A}$-modules, acyclic outside degrees $0, 1$ and $2$ and its Euler characteristic in $K_0(\mathcal{A})$ vanishes.
\item Set $N^\ast (1) = \Hom_\mathcal{A} ( N, \mathcal{A} \otimes_{\Z_p} \Z_p (1))$. There are canonical isomorphisms of $\mathcal{A}$-modules
\[ H^{0}(C_{\mathcal{A}}(N)) \cong H^0(K,N^*(1))\,\,\text{ and }\,\, H^1(C_{\mathcal{A}}(N)) \cong H^1(\mathcal{O}_{K,S},N^*(1))\]
and a canonical exact sequence of ${\mathcal{A}}$-modules
\[ 0 \to H^2(\mathcal{O}_{K,S}, N^*(1)) \to H^2(C_{\mathcal{A}}(N)) \to Y_K(N) \to 0.\]
\item If Hypothesis \ref{hyptorsionfree} is satisfied, then $C_{\mathcal{A}}(N)$ is isomorphic in $D(\mathcal{A})$ to a complex of the form $P\to P$, where $P$ is a finitely generated free $\mathcal{A}$-module and the first term is placed in degree one. 
\item Let $\mathcal{B}$ be a ring that is either finite or a free $\Z_p$-module of finite rank. 
For any homomorphism of rings $\mathcal{A} \to \mathcal{B}$, there exists a canonical isomorphism in $D(\mathcal{B})$ of the form $\mathcal{B}\otimes_{\mathcal{A}}^\DL C_{\mathcal{A}}(N) \cong 
C_{\mathcal{B}}(\mathcal{B}\otimes_{\mathcal{A}}N)$. 
\end{enumerate}
\end{lemma}

\begin{proof} Claims (i), (ii) and (iii) are proved in \cite[Prop.~2.21]{sbA} and we recall only that claims (i) and (iii) depend crucially on the fact that $N$ is free over the Gorenstein algebra $\mathcal{A}$, whilst the isomorphisms in claim (ii) follow directly from the Artin-Verdier Duality Theorem. The isomorphism in claim (iv) is obtained by combining the natural isomorphism in $D(\mathcal{B})$  
\[ \mathcal{B}\otimes_{\mathcal{A}}^\DL \DR\Hom_{\mathcal{A}}(\DR\Gamma_c(\mathcal{O}_{K,S}, N), \mathcal{A} [-3])\\
\cong \mathrm{R}\Hom_{\mathcal{B}}(\mathcal{B}\otimes_{\mathcal{A}}^\DL \DR\Gamma_c(\mathcal{O}_{K,S}, N), \mathcal{B}[-3]) \]
together with the canonical isomorphism 
\[ \mathcal{B}\otimes_{\mathcal{A}}^\DL \DR\Gamma_c(\mathcal{O}_{K,S}, N)\cong \DR\Gamma_c(\mathcal{O}_{K,S},\mathcal{B}\otimes_{\mathcal{A}} N)\]
(as described, for example, in \cite[Prop.\@ 4.2]{Flach00}). \end{proof}

\section{Congruence conjectures}\label{cong conj section}

\subsection{Rank-one motives}\label{rank one motive sec}

Following Deligne \cite[\S 2.4]{deligneconj}, we give the following definition.

\begin{definition}
A motive $M$ (or rather $(M/K,R)$) is of {\it rank one} if ${\rm rank}_R(H_\sigma(M))=1$ for any $\sigma: K \hookrightarrow \CC$ (or equivalently, ${\rm rank}_{R_p}(V_p(M))=1$ for any prime number $p$).
\end{definition}

\begin{example}\label{exrankone}\

\begin{itemize}
\item[(i)] The motive $(h^0(L)(j)/K, e_j^{\pm}\QQ[G])$ in Example \ref{exrank}\,(ii) is of rank one.
\item[(ii)] Let $K$ be an imaginary quadratic field (of class number one) and $E/K$ an elliptic curve with complex multiplication by $\cO_K$. Then the motive $(h^1(E)(1)/K, K)$ is of rank one. More generally, for any algebraic Hecke character $\varphi$ of $K$, one can consider the Hecke motive $(h(\varphi)/K,K)$, which is of rank one.
\item[(iii)] Let $A/K$ be an abelian variety with complex multiplication by an order of a CM field $F$. Then the motive $(h^1(A)(1)/K,F)$ is of rank one.
\end{itemize}
\end{example}

\begin{remark}
Deligne conjectured in \cite[Conj.\@ 8.1\,(iii)]{deligneconj} that every rank-one motive arises from an algebraic Hecke character.
\end{remark}

\begin{definition} \label{character-definition}
Let $M$ be a rank-one motive and $T:=T_p(M)$ be a stable lattice of $V_p(M)$. We define the {\it character associated with $T$} by the composition
\[\chi_T \: G_K \to {\rm Aut}_{\cR_p}(T) \cong \cR_p^\times,\]
where the last isomorphism follows from the fact that ${\rm rank}_{\cR_p}(T)=1$. For each natural number $n$ we similarly define a character
\[
\chi_{T,n} \: G_K \to {\rm Aut}_{\cR/p^n}(T/p^n) \cong (\cR/p^n)^\times.
\]
\end{definition}

\begin{example}
If $M=h^0(K)(1)$ and $T=\ZZ_p(1)$, then $\chi_T$ coincides with the cyclotomic character $\chi_{\rm cyc} \: G_K \to  \ZZ_p^\times.$\end{example}

\subsection{Congruences}\label{cong section}
Let $(M/K,R)$ be a rank-one motive. Fix $p$, $\cR$, $T:=T_p(M) \subset V_p(M)$, and $S$ as in \S\ref{review bk}. We also fix a finite abelian extension $L/K$ unramified outside $S$ with Galois group $G$ and a natural number $n$. In what follows, we always assume Hypothesis \ref{hyp}. Let $r=r(M/K,R)$ be the basic rank and fix an $\cR_p$-basis $b=\{b_1,\ldots,b_r\}$ of $Y_K(T)$.

\newpage
\begin{hypothesis}\label{hypcong}\
\begin{itemize}
\item[(i)] $L$ contains $\overline \QQ^{\ker (\chi_{T,n})}$, (or equivalently, the character $\chi_{T,n}\: G_K \to (\cR/p^n)^\times$ factors through the restriction map $G_K \to G$).
\item[(ii)] There exists an idempotent $\varepsilon \in \ZZ_{p}[G]$ such that
\begin{itemize}
\item[(a)] the ring homomorphism $\ZZ_p[G] \to \cR/p^n$ induced by $\chi_{T,n}$ sends $\varepsilon$ to the identity element of $\cR/p^n$;
\item[(b)] the $\ZZ_p[G]\varepsilon$-module $Y_K(\varepsilon\cdot\ZZ_{p,L/K})$ is free of rank $r$;
\item[(c)] $H^1(\cO_{K,S}, \varepsilon\cdot\ZZ_p(1)_{L/K}) $ is $\ZZ_p$-free.
\end{itemize}
\item[(iii)] There exists a $\ZZ_p[G]\varepsilon$-basis $w=\{w_1,\ldots,w_r\}$ of $Y_K(\varepsilon\cdot\ZZ_{p,L/K})$ such that the map
$$Y_K(\varepsilon\cdot\ZZ_{p,L/K}) \to Y_K(T/p^n)$$
induced by ${\rm tw}_{T,n}$ in Lemma \ref{twist lem} below sends $w$ to the image of $b$.
\end{itemize}
\end{hypothesis}

\begin{lemma}\label{twist lem}
Assume Hypothesis \ref{hypcong}(i) and (ii)(a). Then each choice of an $\cR_p$-basis of $T$ gives rise to an isomorphism
\begin{equation}\label{twist iso}\varepsilon\cdot\ZZ_{p,L/K} \otimes_{\ZZ_p[G], \chi_{T,n}^{-1}} \cR/p^n \cong T/p^n\end{equation}
of $\cR_p[G_K]$-modules, and hence also to a homomorphism
$${\rm tw}_{T,n} \: \varepsilon\cdot \ZZ_{p,L/K} \to T/p^n$$
of $\ZZ_p[ G_K]$-modules.
\end{lemma}

\begin{proof} The given hypotheses imply that the tensor product $\varepsilon\cdot \ZZ_{p,L/K} \otimes_{\ZZ_p[G], \chi_{T,n}^{-1}} \cR/p^n$ is a free $\cR/p^n$-module of rank one upon which $G_K$ acts via the character $\chi_{T,n}$. Since $T/p^n$ is also a free $\cR/p^n$-module of rank one upon which $G_K$ acts via $\chi_{T,n}$, any choice of an $\cR_p$-basis of $T$ induces an isomorphism $ \varepsilon\cdot\ZZ_{p,L/K} \otimes_{\ZZ_p[G], \chi_{T,n}^{-1}} \cR/p^n \cong T/p^n$ of $\cR_p[G_K]$-modules and hence also a composite homomorphism of $\ZZ_p[ G_K]$-modules
$$ \varepsilon\cdot \ZZ_{p,L/K} \to \varepsilon\cdot\ZZ_{p,L/K} \otimes_{\ZZ_p[G], \chi_{T,n}^{-1}} \cR/p^n \cong T/p^n$$
of the required sort.
\end{proof}

\begin{example}
For the motive $(h^0(L)/K,e_j^\pm \QQ[G])$ considered in Example \ref{exrank}(ii), Hypothesis \ref{hypcong}(i) is satisfied if $\mu_{p^n}:=\{\zeta \in \overline \QQ^\times \mid \zeta^{p^n}=1\}\subset L$, and one can take $\varepsilon$ in Hypothesis \ref{hypcong}(ii) to be $\frac{1\pm c}{2}$. Hypothesis \ref{hypcong}(iii) is automatically satisfied in this case.
\end{example}

\begin{lemma}\label{twist modpn}
Assume Hypothesis \ref{hypcong}. Then each choice of $\cR_p$-basis of $T$ gives rise to a canonical homomorphism of $\ZZ_p[G]$-modules
$${\rm tw}_{T,n}^r \: {\bigcap}_{\ZZ_p[G]\varepsilon}^r H^1(\cO_{K,S}, \varepsilon\cdot\ZZ_p(1)_{L/K}) \to {\bigcap}_{\cR/p^n}^r H^1(\cO_{K,S}, T^\ast(1)/p^n),$$
where $G$ acts on the right hand module via $\chi_{T,n}$.
\end{lemma}

\begin{proof} We set $\mathcal{A} := \ZZ_p[G]\varepsilon$, $\cR_n := \cR/p^n$ and $T_n := T/p^n$ and abbreviate the complexes  $C_\mathcal{A}(\varepsilon\cdot\ZZ_{p,L/K})$ and $C_{\cR_n}(T_n)$ defined in \S\ref{gal coh section} to $C_{\varepsilon}$ and $C(T_n)$ respectively. Then by applying the general result of Lemma \ref{compact lemma}\,(iv) to the homomorphism 
$\mathcal{A} \to \cR_n$  induced by $\chi^{-1}_{T,n}$, we find that the isomorphism of $\cR_p[G_K]$-modules 
(\ref{twist iso}) (which depends on a choice of $\cR_p$-basis of $T$) induces a composite morphism 
\[ C_{\varepsilon} \to C_{\varepsilon}\otimes^{\DL}_{\mathcal{A}, \chi^{-1}_{T,n}}\cR_n  \cong C(T_n)\]
in $D(\mathcal{A})$ (in which the first arrow denotes the natural map). 

In addition, since $H^0(K,\ZZ_p(1)_{L/K})$ vanishes, the assumed validity of Hypothesis \ref{hypcong}(ii)(c) implies that the $G_K$-representation $\varepsilon\cdot \ZZ_{p,L/K}$ satisfies Hypothesis \ref{hyptorsionfree}. From Lemma \ref{compact lemma}(iii) it therefore follows that the above displayed morphism is represented by a commutative diagram of the form 
\begin{equation}\label{func diag} \begin{CD} P @> \theta >> P\\
@V\pi_n VV @VV \pi_n V\\
 P_n @> \theta_n >> P_n\end{CD}\end{equation}
in which $P$ is a finitely generated free $\mathcal{A}$-module, $P_n := P\otimes_{\mathcal{A}, \chi^{-1}_{T,n}}\cR_n$, $\theta_n := \theta\otimes {\rm id}$ and $\pi_n$ is the natural projection.

We now set $X:= H^1(\cO_{K,S}, \varepsilon\cdot\ZZ_p(1)_{L/K})$ and $X_n := H^1(\cO_{K,S}, T^\ast(1)/p^n)$. Then, by Lemma \ref{compact lemma}\,(ii), we can identify $X$ and $X_n$ with $\ker(\theta)$ and $\ker(\theta_n)$ respectively. In particular, since $\cR_n$ is self-injective, the general result of \cite[Prop.\@ 2.4]{bss} implies that the inclusion $X_n \subset P_n$ identifies ${\bigcap}_{\cR_n}^rX_n$ with the submodule of
 ${\bigcap}_{\cR_n}^rP_n = {\bigwedge}_{\cR_n}^rP_n$ comprising elements $y$ with the property that $\Phi(y) \in \ker(\theta_n)$ for all $\Phi$ in ${\bigwedge}_{\cR_n}^{r-1}\Hom_{\cR_n}(P_n,\cR_n)$.

To obtain a homomorphism ${\rm tw}_{T,n}^r$ of the required sort it is thus enough to show that the latter condition is satisfied by any element in the image of the composite homomorphism
\begin{equation}\label{pre-twist} {\bigcap}_{\mathcal{A}}^rX \xrightarrow{\iota} {\bigcap}_{\mathcal{A}}^rP = {\bigwedge}_{\mathcal{A}}^rP \to \left({\bigwedge}_{\mathcal{A}}^rP\right)\otimes_{\mathcal{A}, \chi^{-1}_{T,n}}\cR_n \cong {\bigwedge}_{\cR_n}^rP_n\end{equation}
in which $\iota$ is induced by the inclusion $X \subset P$ and the second arrow is the natural projection. Hence, given the commutativity of (\ref{func diag}), and the fact that ${\bigwedge}_{\cR_n}^{r-1}\Hom_{\cR_n}(P_n,\cR_n)$ is generated over $\cR_n$ by the image of the natural map
\[ {\bigwedge}_{\mathcal{A}}^{r-1}\Hom_{\mathcal{A}}(P,\mathcal{A}) \to \left({\bigwedge}_{\mathcal{A}}^{r-1}\Hom_{\mathcal{A}}(P,\mathcal{A})\right)
\otimes_{\mathcal{A}, \chi_{T,n}^{-1}}\cR_n = {\bigwedge}_{\cR_n}^{r-1}\Hom_{\cR_n}(P_n,\cR_n),\]
it is enough to prove $\theta(\Theta(\iota(x))) = 0$ for all $x\in {\bigcap}_{\mathcal{A}}^rX$ and $\Theta\in {\bigwedge}_{\mathcal{A}}^{r-1}\Hom_{\mathcal{A}}(P,\mathcal{A})$. But this is true since $\Theta(\iota(x)) = \Theta'(x)$ for an element $\Theta'$ of ${\bigwedge}_{\mathcal{A}}^{r-1}\Hom_{\mathcal{A}}(X,\mathcal{A})$ and any element of the latter group maps ${\bigcap}_{\mathcal{A}}^rX$ to ${\bigcap}_{\mathcal{A}}^1X = X = \ker(\theta)$.

This shows (\ref{pre-twist}) induces a map ${\bigcap}_{\mathcal{A}}^rX\to {\bigcap}_{\cR_n}^rX_n$ and it is straightforward to show that this construction is independent of the representative complex $P\xrightarrow{\theta}P$ fixed above. \end{proof}

Next we note that, since $H^1(\cO_{K,S}, T^\ast(1))$ is $\ZZ_p$-free (under Hypothesis \ref{hyptorsionfree}) and $\cR_p$ is Gorenstein, the argument of \cite[(11)]{bss} implies the existence of a natural homomorphism
$${\bigcap}_{\cR_p}^r H^1(\cO_{K,S},T^\ast(1))\to {\bigcap}_{\cR/p^n}^r H^1(\cO_{K,S}, T^\ast(1)/p^n); \ a \mapsto \overline a.$$

We can now formulate a precise congruence conjecture concerning Bloch-Kato elements.

\begin{conjecture}[Congruence Conjecture, {${\rm CC}(M/K,R,L,n)$}]\label{cc conj}
Assume Hypotheses \ref{hyp}, \ref{hyptorsionfree} and \ref{hypcong} and the validity of the Integrality Conjecture (Conjecture \ref{ic}) for both $\varepsilon\cdot \ZZ_{p,L/K}$ and $T$. Then in the finite module ${\bigcap}_{\cR/p^n}^r H^1(\cO_{K,S}, T^\ast(1)/p^n)$
 one has
$${\rm tw}_{T,n}^r (\eta_S^w(\varepsilon\cdot\ZZ_{p,L/K})) = \overline {\eta_S^b(T)}.$$
\end{conjecture}

\begin{remark}\label{remark symbol}
Conjecture \ref{conj} is formulated for the data $(M/K,R,p,\cR, T, S,b, L, n)$. Since the data $(p,\cR,T,S,b)$ is often fixed when $(M/K,R)$ is given, we indicate the conjecture by the symbol ${\rm CC}(M/K,R,L,n)$.
\end{remark}

\begin{remark}
 Since $\eta_S^w(\varepsilon \cdot \ZZ_{p,L/K})$ is the $\varepsilon$-component of the Rubin-Stark element (see \S \ref{review stark}), Conjecture ${\rm CC}(M/K,R,L,n)$ predicts a precise congruence relation between Rubin-Stark elements and the Bloch-Kato elements for $M$. More generally, for another motive $M'$, it is possible to formulate a congruence between Bloch-Kato elements for $M'$ and $M \otimes M'$.
\end{remark}

\subsection{Soul\'e-Stark elements} \label{soule-stark-section}
  To formulate a `limit version' of the Congruence Conjecture we fix a rank-one motive $(M/K,R)$ and data $p,\cR,T,S, r, b$ be as in \S\ref{cong section}. (We do not fix $L$ and $n$ in this subsection.) We also define fields 
$$L_n:=\overline \QQ^{\ker(\chi_{T,n})}\,\,\text{ and }\,\, L_\infty:=\bigcup_n L_n =\overline \QQ^{\ker(\chi_{T})} ,$$
and use the associated algebras 
$$\Lambda_n:=\ZZ_p[\Gal(L_n/K)]\,\, \text{ and }\,\, \Lambda:=\ZZ_p\llbracket \Gal(L_\infty/K)\rrbracket.$$
For each  $\ZZ_p \llbracket G_K \rrbracket$-module $X$ we write $X_{L_\infty/K}$ for the inverse limit $\varprojlim_n
 X_{L_n/K}$, where the transition morphisms are the natural projection maps 
\[ X_{L_{n+1}/K} \cong \Lambda_{n+1}\otimes_{\Z_p}X \to \Lambda_{n}\otimes_{\Z_p}X \cong X_{L_{n}/K}.\]

\begin{hypothesis}\label{hypsoule}\
\begin{itemize}
\item[(i)] There exists an idempotent $\varepsilon \in \Lambda$ such that
\begin{itemize}
\item[(a)] the ring homomorphism $\!\Lambda\!\to\!\! \cR_p$ induced by $\!\chi_T$ sends $\varepsilon\!$ to the identity element of $\cR_p$,
\item[(b)] the $\Lambda \varepsilon$-module $Y_K(\varepsilon\cdot \ZZ_{p,L_\infty/K})$ is free of rank $r$,
\item[(c)] $H^1(\cO_{K,S},\varepsilon\cdot \ZZ_p(1)_{L_n/K})$ is $\ZZ_p$-free for every $n$.
\end{itemize}
\item[(ii)] There exists a $\Lambda\varepsilon$-basis $w=\{w_1,\ldots,w_r\}$ of $Y_K(\varepsilon\cdot \ZZ_{p,L_\infty/K})$ such that the map
$$Y_K(\varepsilon\cdot \ZZ_{p,L_\infty/K}) \to Y_K(T)$$
induced by ${\rm tw}_T$ in Lemma \ref{lemtwist2} below sends $w$ to $b$.
\end{itemize}
\end{hypothesis}

\begin{lemma}\label{lemtwist2}
Assume Hypothesis \ref{hypsoule}\,(i)\,(a). Then there is an isomorphism
$$ \varepsilon\cdot \ZZ_{p,L_\infty/K} \otimes_{\Lambda, \chi_T^{-1}} \cR_p \cong T.$$
In particular, there is a natural map ${\rm tw}_T: \varepsilon\cdot \ZZ_{p,L_\infty/K} \to T.$ \end{lemma}

\begin{proof}
This is proved in the same way as Lemma \ref{twist lem}.
\end{proof}

\begin{remark}
Hypothesis \ref{hypsoule} implies Hypothesis \ref{hypcong} for $L=L_n$ for every $n$.
\end{remark}

\begin{definition}\label{def soule stark}
Assume Hypotheses \ref{hyp}, \ref{hyptorsionfree} and \ref{hypsoule} and that the Integrality Conjecture (Conjecture \ref{ic}) is valid for $\varepsilon\cdot\ZZ_{p,L_n/K}$ for every $n$. We define a {\it rank $r$ Soul\'e twist} to be the map
\begin{eqnarray*}
{\rm tw}_T^r:=\varprojlim_n {\rm tw}_{T,n}^r: \varprojlim_n \varepsilon\cdot {\bigcap}_{\Lambda_n}^r H^1(\cO_{K,S}, \ZZ_p(1)_{L_n/K}) &\longrightarrow & \varprojlim_n {\bigcap}_{\cR/p^n}^r H^1(\cO_{K,S},T^\ast(1)/p^n)\\
&\cong & {\bigcap}_{\cR_p}^r H^1(\cO_{K,S},T^\ast(1)),
\end{eqnarray*}
where the isomorphism is by \cite[Lem.~2.4]{ghaleo}. We define the {\it Soul\'e-Stark element} for $(M/K,R)$ by setting
\begin{eqnarray*}
\beta_S^b(T):= {\rm tw}_{T}^r \left(\varprojlim_n \eta_S^w(\varepsilon \cdot\ZZ_{p,L_n/K}) \right) \in {\bigcap}_{\cR_p}^r H^1(\cO_{K,S},T^\ast(1)).
\end{eqnarray*}
One checks that this element is independent of the choice of $w$ in Hypothesis \ref{hypsoule}\,(ii).
\end{definition}

\begin{example}\label{exsoule}\

\begin{enumerate}[label=(\roman*)]
\item Let $f$ be a positive integer divisible by $p$ and write $\zeta_f$ for a primitive $f$-th root of unity. We set $G:=\Gal(\QQ(\zeta_f)/\QQ).$ The Soul\'e-Stark element for the motive $(h^0(\QQ(\zeta_f))(j)/\QQ,e_j^+ \QQ[G])$ is the Deligne-Soul\'e cyclotomic element
$$e_j^+ c_{1-j}(\zeta_f) \in H^1(\QQ(\zeta_f),\ZZ_p(1-j))$$
(see \cite[Def.\@ 3.1.2]{HK}). 
\item Let $L/K$ be a finite abelian extension with Galois group $G$, and assume that $K$ is totally real and $L $ is CM. Then for any integer $j\neq 1$ the Soul\'e-Stark element for the motive $(h^0(L)(j)/K, e_j^-\QQ[G])$ is the image of the Stickelberger element $\theta_{L_\infty/K,S}(0)$ (see \cite[\S4B]{bks2}) under the map
$$\Lambda=\ZZ_p \llbracket \Gal(L_\infty/K)\rrbracket \to \ZZ_p[G]; \ \sigma \mapsto \chi_{\rm cyc}(\sigma)^j \overline \sigma,$$
where $\overline \sigma \in G$ denotes the image of $\sigma \in \Gal(L_\infty/K)$.
(Here we implicitly choose $\Sigma$ in Remark \ref{sigma} so that $\theta_{L_\infty/K,S}(0)$ lies in $\Lambda$.)
\end{enumerate}
\end{example}

We conjecture that the Soul\'e-Stark element coincides with the Bloch-Kato element.

\begin{conjecture}[Soul\'e-Stark Conjecture, {${\rm SS}(M/K,R)$}]\label{conj}
$$\eta_S^b(T)=\beta_S^b(T).$$
\end{conjecture}

\begin{remark}
Conjecture \ref{conj} is formulated for the data $(M/K,R,p,\cR, T, S,b)$. We omit the data $(p,\cR,T,S,b)$ for the same reason as Remark \ref{remark symbol}.
\end{remark}

\begin{remark}
Conjecture ${\rm SS}(M/K,R)$ implies Conjecture ${\rm CC}(M/K,R,L_n,n)$ for all $n$.
\end{remark}

\subsection{Connection to the eTNC}

In this section we first review the main conjecture of higher-rank Iwasawa theory for $\GG_m$ that is formulated by Kurihara and the first and third authors in \cite{bks2}. We then explain how this conjecture can be combined with the Soul\'e-Stark Conjecture to give a precise generalization to rank-one motives of the main strategy used by Huber and Kings in \cite{HK} to prove the Tamagawa Number Conjecture for Tate motives over abelian extensions of $\QQ$.

\subsubsection{The higher-rank main conjecture} 

We assume for the moment that $L_\infty$ is any Galois extension of $K$ for which $\Gal(L_\infty / K)$ is isomorphic to $\Z_p^d \times \Delta$ for some natural number $d$ and finite abelian group $\Delta$ and, in addition, no finite place of $K$ splits completely in $L_\infty$. 

We write $\Omega(L_\infty)$ for the set of finite extensions of $K$ in $L_\infty$. 
 We set $\Lambda := \ZZ_p\llbracket\Gal(L_\infty/K)\rrbracket$ and $\Lambda_F := \ZZ_p[\Gal(F/K)]$ for each $F$ in $\Omega(L_\infty)$. We also fix an idempotent $\varepsilon$ of $\Lambda$ that validates the conditions (i)\,(b) and (i)\,(c) in Hypothesis \ref{hypsoule}. We write $Q(\Lambda)$ for the total quotient ring of $\Lambda$. 
 
Now, since no finite place of $K$ splits completely in $L_\infty$, the 
$\Lambda$-module $H^2(\cO_{K,S},\ZZ_p(1)_{L_\infty/K})$ is torsion. This fact combines with the results of Lemma \ref{compact lemma}\,(i) and (ii) (with $\mathcal{A} = \Lambda_F\varepsilon$ and $N = \varepsilon\cdot \ZZ_{p,F/K}$ for each $F$ in $\Omega(L_\infty)$) and the assumed validity of Hypothesis \ref{hypsoule}\,(i)\,(b) to imply that the complex $Q(\Lambda)\varepsilon\otimes_{\Lambda}\rgamma(\cO_{K,S},\ZZ_p(1)_{L_\infty/K})$ is acyclic outside degree one and that its cohomology in degree one is free of rank $r$. Hence, under the present hypotheses, there exists a canonical isomorphism of $Q(\Lambda)$-modules 
\[  Q(\Lambda)\varepsilon\otimes_{\Lambda} {\det}_\Lambda^{-1}(\rgamma(\cO_{K,S},\ZZ_p(1)_{L_\infty/K})) \cong  Q(\Lambda) \varepsilon \otimes_\Lambda {\bigwedge}_\Lambda^r H^1(\cO_{K,S},\ZZ_p(1)_{L_\infty/K}).\]

We further note that, by the argument in \cite[Thm.~3.4\,(1) and Lem.~3.5]{bks2}, for any $\Lambda \varepsilon$-order $\mathfrak{A}$ in $Q(\Lambda) \varepsilon$, the above isomorphism restricts to give an 
injective homomorphism of $\mathfrak{A}$-modules of the form 
\begin{multline*}\pi_\infty \:  \mathfrak{A} \otimes_{\Lambda}  {\det}_\Lambda^{-1}(\rgamma(\cO_{K,S},\ZZ_p(1)_{L_\infty/K})) \to \mathfrak{A} \otimes_{\Lambda}  \varprojlim_{F\in \Omega(L_\infty)} {\bigcap}_{\Lambda_F}^r H^1(\cO_{K,S}, \ZZ_p(1)_{F/K}).
\end{multline*}

Setting   
\[ \eta_{L_\infty/K,S}^w := (\eta_S^w(\varepsilon\cdot\ZZ_{p,F/K}))_{F\in \Omega(L_\infty)}\in \varprojlim_{F\in \Omega(L_\infty)}\CC_p\otimes_{\ZZ_p}\varepsilon\cdot {\bigwedge}_{\Lambda_F}^r H^1(\cO_{K,S}, \ZZ_p(1)_{F/K}),\]
we can now recall the formulation of the $\varepsilon$-component of the Iwasawa Main Conjecture for $\GG_m$ for $(L_\infty/K,S)$. (In fact, since we do not assume that $d=1$, the formulation we give here is actually slightly more general than in loc. cit.)

\begin{conjecture}[{Kurihara et al.\@ \cite[Conj.\@ 3.1]{bks2}}] \label{IMCgeneral} 
Assume the above hypotheses and fix a $\Lambda\varepsilon$-order $\mathfrak{A}$ in $Q(\Lambda)\varepsilon$.  Then there exists an $\mathfrak{A}$-basis 
\[ \fz_{L_\infty/K,S}^w\in  \mathfrak{A} \otimes_{\Lambda} {\det}_\Lambda^{-1}(\rgamma(\cO_{K,S},\ZZ_p(1)_{L_\infty/K}))\]
for which one has 
\[
\pi_\infty ( \fz_{L_\infty/K,S}^w) = \eta_{L_\infty/K,S}^w.
\]
\end{conjecture}

\subsubsection{The strategy of Huber and Kings} We shall now prove the following result. 

\begin{theorem}\label{descent} Fix a field $L_\infty = \bigcup_{n}L_n$ as in \S\ref{soule-stark-section}. 
Assume that $\Gal(L_\infty/K)$ is isomorphic to $\ZZ_p^d \times \Delta$ for some natural number $d$ and finite abelian group $\Delta$ and that no finite place in $S$ splits completely in $L_\infty$. Assume also that Hypotheses \ref{hyp}, \ref{hyptorsionfree} and \ref{hypsoule} are valid and that, for every $n$, the Integrality Conjecture (Conjecture \ref{ic}) is valid for the data $\varepsilon\cdot\ZZ_{p,L_n/K}$. Then the equivariant Tamagawa Number Conjecture for the pair $(M,\cR_p)$ is valid whenever the following three conditions are satisfied: 
\begin{enumerate}[label=(\alph*)]
\item The higher-rank Iwasawa Main Conjecture for $\GG_m$ (Conjecture \ref{IMCgeneral}) is valid with $\mathfrak{A} = \Lambda\varepsilon$; 
\item The Soul\'e-Stark Conjecture ${\rm SS}(M/K,R)$ is valid; 
\item $H^2(\cO_{K,S}, V^\ast(1))$ vanishes. 
\end{enumerate}
\end{theorem}

\begin{remark}\label{remark descent} Condition (c) in Theorem \ref{descent} is predicted to be satisfied in all but a few exceptional cases (see Jannsen \cite[Conj.\@ 1]{jannsen-1}). However, if $T = \tilde\varepsilon \cdot \ZZ_{p, F/K}$ for a finite abelian extension $F$ of $K$ and an idempotent $\tilde\varepsilon$ of $\ZZ_p[\Gal(F/K)]$, then class field theory implies that (c) is satisfied only if $\tilde\varepsilon$ annihilates the submodule $X$ of the free $\ZZ_p$-module on the places of $F$ above $S\setminus S_\infty(K)$ comprising elements whose coefficients sum to zero. In the case $T = \tilde\varepsilon \cdot \ZZ_{p, F/K}$ and $\tilde\varepsilon\cdot X\not= 0$ the result of \cite[Thm.\@ 5.2]{bks2} gives an alternative strategy for proving the equivariant Tamagawa Number Conjecture for $(h^0(F),\ZZ_p[\Gal(F/K)]\tilde\varepsilon)$ that involves the Mazur-Rubin-Sano Conjecture.
\end{remark}

\begin{remark}\label{extended field remark} The field $L_\infty = \bigcup_n L_n$ in Theorem \ref{descent} is determined by $T$. However, the argument given below will show that the same result is valid if in the statement of  Theorem \ref{descent} one replaces $L_\infty$, respectively $L_n$, by any abelian extension $L_\infty' = \bigcup_n L_n'$, respectively $L_n'$,  of $K$ with the property that $L_n \subseteq L_n'$ for each $n$ and $\Gal(L_\infty'/K)$ is isomorphic to $\ZZ_p^d\times \Delta$ for some natural number $d$ and finite abelian group $\Delta$. 
\end{remark}

The proof of Theorem \ref{descent} will now occupy the remainder of this section. We start by making the following technical observation. 

\begin{lemma}\label{etnc}
Assume Hypotheses \ref{hyp} and \ref{hyptorsionfree} and that $H^2(\cO_{K,S}, V^\ast(1))$ vanishes. Then we have a canonical isomorphism
$$\pi_T \: \QQ_p\otimes_{\ZZ_p}{\det}_{\cR_p}^{-1}(\rgamma (\cO_{K,S}, T^\ast(1))) \cong \QQ_p\otimes_{\ZZ_p} {\bigwedge}_{\cR_p}^r H^1(\cO_{K,S},T^\ast(1))$$
and the equivariant Tamagawa Number Conjecture for $(M,\cR_p)$ holds if and only if there exists an $\cR_p$-basis
$$\fz_S^b(T) \in {\det}_{\cR_p}^{-1}(\rgamma (\cO_{K,S}, T^\ast(1)))$$
such that
$$\pi_T (\fz_S^b(T))=\eta_S^b(T). $$
\end{lemma}

\begin{proof} The vanishing of $H^2(\cO_{K,S}, V^\ast(1))$ implies that $\rgamma (\cO_{K,S}, V^\ast(1))$ is acyclic outside degree one and also combines with Hypothesis \ref{hyp} and the result of Lemma \ref{compact lemma}\,(i) and (ii) to imply $H^1(\cO_{K,S},V^\ast(1))$ is a free $(\QQ_p\otimes_{\ZZ_p} \cR_p)$-module of rank $r$. The first claim follows directly from this. The second claim then follows from the definition of Bloch-Kato elements and the precise formulation of the equivariant Tamagawa Number Conjecture. 
\end{proof}

We note that the vanishing of $H^2(\cO_{K,S}, V^\ast(1))$ plays an essential role in the construction of the isomorphism  $\pi_T$ in 
 Lemma \ref{etnc}. We will also later use the fact (that is verified by the argument of \cite[Prop.\@ 2.21]{sbA}) that this map  $\pi_T$ restricts to give a homomorphism of $\cR_p$-modules (that we denote by the same symbol) of the form 
\[
\pi_T \: {\det}_{\cR_p}^{-1}(\rgamma (\cO_{K,S}, T^\ast(1))) \to {\bigcap}_{\cR_p}^r H^1(\cO_{K,S},T^\ast(1)).
\]

In addition, just as in Lemma \ref{lemtwist2}, there exists an isomorphism
\[\varepsilon\cdot \ZZ_p(1)_{L_\infty/K} \otimes_{\Lambda, \chi_T} \cR_p \cong T^\ast(1),\]
which induces a map
\[{\rm tw}_{T}^{\rm det} \: \varepsilon\cdot{\det}_\Lambda^{-1}(\rgamma(\cO_{K,S},\ZZ_p(1)_{L_\infty/K})) \to  {\det}_{\cR_p}^{-1}(\rgamma (\cO_{K,S}, T^\ast(1))).\]

The relation between the maps $\pi_T, {\rm tw}_{T}^{\rm det}$ and ${\rm tw}_{T}^r$ is described by the following result.  

\begin{lemma}\label{comm}
Under the hypotheses of Theorem \ref{descent}, the following diagram commutes:
\[
\xymatrix{
\varepsilon\cdot{\det}_\Lambda^{-1}(\rgamma(\cO_{K,S},\ZZ_p(1)_{L_\infty/K})) \ar[rr]^{\quad \quad{\rm tw}_T^{\rm det}} \ar[d]_{\pi_\infty}& &  {\det}_{\cR_p}^{-1}(\rgamma (\cO_{K,S}, T^\ast(1))) \ar[d]^{\pi_T} \\
\varprojlim_n \varepsilon \cdot{\bigcap}_{\Lambda_n}^r H^1(\cO_{K,S}, \ZZ_p(1)_{L_n/K}) \ar[rr]_{\quad \quad    {\rm tw}_{T}^r}& &{\bigcap}_{\cR_p}^r H^1(\cO_{K,S},T^\ast(1)).
}
\]
\end{lemma}

\begin{proof}
This is proved by a standard argument. See, for example, the argument of Tsoi \cite[Cor.\@ 4.9]{ghaleo}.
\end{proof}

We can now give a proof of Theorem \ref{descent}. To do this we note, firstly, that the assumed validity of Conjecture \ref{IMCgeneral} with $\mathfrak{A} = \Lambda\varepsilon$ implies the existence of a $\Lambda\varepsilon$-basis
\begin{eqnarray*}
\fz_{L_\infty/K,S}^w \in \varepsilon\cdot{\det}_\Lambda^{-1}(\rgamma(\cO_{K,S},\ZZ_p(1)_{L_\infty/K}))
\end{eqnarray*}
such that $\pi_\infty(\fz_{L_\infty/K,S}^w)=\eta_{L_\infty/K,S}^w$. The element 
$$\fz_S^b(T):= {\rm tw}_T^{\rm det}(\fz_{L_\infty/K,S}^w) $$
is therefore an $\cR_p$-basis of ${\det}_{\cR_p}^{-1}(\rgamma(\cO_{K,S},T^\ast(1)))$. 

In addition, since the very definition of $\beta_S^b(T)$ implies that it is equal to 
$$\beta_S^b(T)= {\rm tw}_{T}^r (\eta_{L_\infty/K,S}^w) \in  {\bigcap}_{\cR_p}^r H^1(\cO_{K,S},T^\ast(1)),$$
the commutative diagram in Lemma \ref{comm} implies that 
$$\pi_T (\fz_S^b(T))=\beta_S^b(T). $$

From Conjecture ${\rm SS}(M/K,R)$, we can therefore deduce that 
$$\pi_T (\fz_S^b(T))=\eta_S^b(T) $$
and, by Lemma \ref{etnc}, this equality is equivalent to the validity of the equivariant Tamagawa Number Conjecture for the pair $(M,\cR_p)$. \medskip \\
This completes the proof of Theorem \ref{descent}.  \qed

\section{Iwasawa Main Conjectures for $\mathbb{G}_m$}\label{mc Gm section}

In this section we provide new evidence for higher-rank Iwasawa Main Conjectures for $\mathbb{G}_m$ including, in particular, 
 Conjecture \ref{IMCgeneral}.

\subsection{Statement of the main result}

Let $L_\infty / K$ be an abelian extension in which no finite place splits completely and which is such that $\Gal ( L_\infty / K) \cong \Gamma \times G$ for a finite abelian group $G$ and $\Gamma \cong \Z_p^d$ with $d > 0$. Put $L =  L_\infty^\Gamma$. Let $P \subseteq G$ be the $p$-Sylow subgroup of $G$ and set $\Delta =  G / P$. We fix a character $\chi \: \Delta \to \overline{\Q_p}^\times$ and, following \cite[Hyp. 2.9 and Hyp. 3.1]{bss3}, we consider the following hypothesis on this character. 

\begin{hypothesis} \label{bss-hypothesis} The character $\chi$ satisfies each of the following conditions: 
\begin{enumerate}[label=(\roman*)]
\item $\chi \not \in \{ 1, \omega \}$, where $\omega$ is the Teichm\"uller character of $K$;
\item $\chi^2 \neq \omega$ if $ p = 3$;
\item $\chi (v) \neq 1$ for every $v \in S_\text{ram} ( L / K) \cup S_p (K)$;
\item $r : = | \{ v \in S_\infty (K) \mid \chi (v) = 1 \} | > 0$.
\end{enumerate}
\end{hypothesis}

Let $L_\chi$ be the subfield of $L$ cut out by the character $\chi$ and,
for any subfield $F$ of $L_\infty / K$, denote $\Gal (F / K)$ by $\cG_F$. 

We write $S$ for the finite set of places $S_\infty (K) \cup S_\text{ram} ( L/ K) \cup S_p (K)$ and set 
\[
U_{L_\infty} := \varprojlim_F (\Z_p \otimes_\Z \bigO_{F, S}^\times ) 
\quad \text{ and } \quad
\Cl (L_\infty) := \varprojlim_F \Cl (F),
\]
where in both limits $F$ ranges over all finite subfields of $L_\infty / L$ and $\Cl (F)$ denotes the $p$-part of the ideal class group of $F$. We take $\Lambda = \Z_p [\im \chi] \llbracket \Gamma \rrbracket [P]$ to be the relevant Iwasawa algebra. \medskip \\
As before we fix a basis $b$ of $Y_K (\Z_{p, L_\infty / K })$ and, by abuse of notation, we will denote the induced basis of $Y_K ( \Z_{p, F / K})$ for a subfield $F$ of $L_\infty / K$ by $b$ as well. We denote the Rubin-Stark element relative to this choice of data by $\eta^b_{F / K, S}$ (cf.\@ \S \ref{review stark}). \\
Finally, in the sequel, we will also use the following general notation: the \textit{$\chi$-component} of a $\Delta$-module $X$ is defined to be 
\[ X_\chi := X \otimes_{\Z [\Delta]} \Z_p [\im \chi]\]
 with $\Delta$ acting on $\Z_p [\im \chi]$ as $\sigma \cdot x = \chi (\sigma) \cdot x$ and for any element $a \in X$ we set $a^\chi := a \otimes 1 \in X_\chi$. For any $\Lambda$-module $X$ and natural number $i$ we moreover set 
 \[ \EE^i_\Lambda (X) := \Ext^i_\Lambda (X, \Lambda).\] 

We can now state the main result of this section. 

\begin{theorem} \label{classical-imc-result}
Assume the Rubin-Stark Conjecture is valid for all abelian extensions of $K$ and that $\chi$ validates Hypothesis \ref{bss-hypothesis}. Then the following claims are also valid. 
\begin{enumerate}[label=(\roman*)]
\item The norm-coherent family $(\eta^b_{F / K, S})_F$ of Rubin-Stark elements, where $F$ ranges over all finite extensions of $L$ in $L_\infty$, defines a canonical element $\eta^{b, \chi}_{L_\infty / K, S}$ of $\bidual^r_{\Lambda} U_{L_\infty, \chi}$. For this element one has an inclusion 
\[
 \Fitt^0_\Lambda \big ( \Cl (L_\infty)_\chi \big) \subseteq \im \big ( \eta^{b, \chi}_{L_\infty / K, S} \big )
\]
with pseudo-null cokernel. 
\item For the canonical functorial homomorphism 
\[ \kappa \: \EE^1_\Lambda \big ( \big ( \bidual^r_\Lambda U_{L_\infty, \chi} \big ) / ( \Lambda \cdot \eta^{b, \chi}_{L_\infty / K, S} ) \big) \to \EE^1_\Lambda \big (  \bidual^r_\Lambda U_{L_\infty, \chi}\big ) \]
one has  
\begin{equation*} 
\Fitt^0_\Lambda (\ker (\kappa)) =  \im \big ( \eta^{b, \chi}_{L_\infty / K, S} \big )^{\ast \ast}.
\end{equation*}
In particular, if $\wp$ is any prime ideal of $\Lambda$ of height at most one, then one has 
\begin{equation} \label{classical-imc-2}
\Fitt^0_{\Lambda} \big(\EE^1_\Lambda \big ( \big ( \bidual^r_\Lambda U_{L_\infty, \chi} \big ) / 
( \Lambda \cdot \eta^{b, \chi}_{L_\infty / K, S} ) \big)\big)_\wp = \Fitt^0_{\Lambda} \big ( \Cl (L_\infty)_{\chi} \big)_\wp.
\end{equation}
\end{enumerate}
\end{theorem}

\begin{remark} In certain natural situations, the equality (\ref{classical-imc-2}) has a more explicit interpretation. For example, 
if we assume that $G$ contains no element of order $p$, then (since characteristic ideals over power series rings are determined by their localisations at height one prime ideals),  (\ref{classical-imc-2}) implies that     
\[
\mathrm{char}_\Lambda \Big ( \big (\bidual^r_\Lambda U_{L_\infty, \chi} \big) / ( \Lambda \cdot \eta^{b, \chi}_{L_\infty / K, S} ) \Big)
= \mathrm{char}_\Lambda \big ( \Cl (L_\infty)_\chi \big).
\]
In regard of the latter equality, we recall that main conjectures in this more classical style have been proved (under certain additional hypotheses that include the assumed validity of Leopoldt's Conjecture) by B\"uy\"ukboduk and by B\"uy\"ukboduk and Lei in \cite{buyukboduk-1}, \cite{buyukboduk-2} and \cite[Thm.~7.7]{BL1}.
\end{remark}

\subsection{The proof of Theorem \ref{classical-imc-result}}

At the outset we note that, since $\chi$ validates Hypothesis \ref{bss-hypothesis}, 
the representation $\Z_{p, F / K} (1)(\chi^{-1})$ satisfies Hypothesis \ref{hyptorsionfree} for all finite intermediate extensions of $L_\infty / K$. Passing to the limit over $F$ in Lemma \ref{compact lemma}\,(iii) we therefore have that the complex
\[
C_{\infty, \chi} := \rgamma (\bigO_{K, S}, \Z_{p, L_\infty / K} (1)(\chi^{-1})) \oplus Y_K (\Z_{p, L_\infty / K})_\chi [-2]
\]
is isomorphic in $D(\Lambda)$ to a complex of the form 
\begin{equation}\label{quad rep} 
P \stackrel{\theta}{\longrightarrow} P, 
\end{equation}
where $P$ is a free $\Lambda$-module of finite rank and the first term is placed in degree one. In addition, 
Lemma \ref{compact lemma}\,(ii) combines with Kummer theory and class field theory to
give a canonical isomorphism of $\Lambda$-modules
\begin{equation}\label{can iso 1} \ker(\theta) \cong H^1 (C_{\infty, \chi}) \cong U_{L_\infty, \chi}\end{equation}
and 
a canonical split-exact sequence 
\begin{equation}\label{can iso 2}
0 \longrightarrow \Cl (L_\infty)_\chi \longrightarrow {\rm coker}(\theta) \longrightarrow Y_K ( \Z_{p, L_\infty / K})_\chi \longrightarrow 0,
\end{equation}
respectively. 
Next, we note that for any finite abelian extension $F$ of $L$ contained in $L_\infty$, the complex 
\[ C_{F, \chi} := \rgamma ( \bigO_{K, S}, \Z_{p, F / K} (1)(\chi^{-1})) \oplus Y_K ( \Z_{p, F / K})_\chi\]
is naturally isomorphic to $C_{\infty, \chi} \otimes^\mathbb{L}_\Lambda \Z_p [\cG_F]$ and hence isomorphic in $D(\Z_p [\cG_F]_\chi)$ to the complex
\begin{equation}\label{can rep 2} P_F \xrightarrow{\theta_F} P_F,\end{equation}
where $P_F$ denotes the $\Z_p [\cG_F]_\chi$-free module $P \otimes_{\Lambda} \Z_p [\cG_F]$, the first term is placed in degree one and $\theta_F$ denotes the endomorphism of $P_F$ induced by $\theta$.\\
In particular, since the module $( \Z_p \otimes_\Z \bigO_{F, S}^\times )_\chi = H^1(C_{F,\chi})$ identifies with $\ker(\theta_F)$, the above discussion allows us to apply a general observation of Sakamoto (see \cite[Lem.~B.15]{Sakamoto20}) in order to obtain a canonical identification
\[
 \varprojlim_F \bidual^r_{\Z_p [\cG_F]_\chi} ( \Z_p \otimes_\Z \bigO_{F, S}^\times )_\chi = \bidual^r_{\Lambda} U_{L_\infty, \chi}.
\]
Given this identification, the first assertion in claim (i) is therefore clear. \medskip \\
We next claim that, for any element $a = (a_F)_F$ of $\bidual^r_\Lambda U_{L_\infty, \chi}$, there is an inclusion of ideals of $\Lambda$ 
\begin{equation} \label{limit-images}
 \varprojlim_F \im (a_F) \subseteq \im (a)
\end{equation}
with pseudo-null cokernel.
To see this, we recall that
\begin{equation} \label{images}
\im (a) \supseteq \big \{ f (a) \mid f \in \exprod^r_\Lambda P^\ast \big \}
\quad \text{ and } \quad
\im (a_F) = \big \{ f (a) \mid f \in \exprod^r_{\Z_p [\cG_F]_\chi} P_F^\ast \big \},
\end{equation} 
where the first inclusion has pseudo-null cokernel
(see the proof of \cite[Lem.~2.7\,(c)]{BullachDaoud}). Note that, since the $\Lambda$-module $P^\ast = \varprojlim_F P_F^\ast$ is free, the module $\exprod^r_\Lambda P^\ast$ identifies with the limit $\varprojlim_F \exprod^r_{\Z_p [\cG_F]_\chi} P_F^\ast$ and so we have an equality
\begin{equation} \label{limit-comp}
 \big \{ f (a) \mid f \in \exprod^r_\Lambda P^\ast \big \} = \varprojlim_F  \big \{ f (a) \mid f \in \exprod^r_{\Z_p [\cG_F]_\chi} P_F^\ast \big \}.
\end{equation} 
This combines with (\ref{images}) to imply the claimed inclusion (\ref{limit-images}).
\medskip \\
The key point now is that the compatibility in (\ref{limit-images}) allows us to deduce from the result  \cite[Cor.~3.6]{bss3} of Sakamoto and the second and third authors that there is an equality
\begin{equation} \label{bss-limit-step-1}
\varprojlim_F \im \big ( \eta^{b, \chi}_{F / K, S} \big ) = \varprojlim_F \Fitt^0_{\Z_p [\cG_F]} ( \Cl (F))_\chi.
\end{equation}
To prove the second assertion in claim (i) we therefore need to show that the right hand side of (\ref{bss-limit-step-1}) agrees with $\Fitt^0_\Lambda \big ( \Cl (L_\infty)_\chi \big)$. 
However, as the complexes $C_{\infty, \chi}$ and $C_{F, \chi}$ are respectively isomorphic to the complexes 
 (\ref{quad rep}) and (\ref{can rep 2}), the general result of \cite[Lem.~A.7]{sbA} applies in this case to imply the existence of an element $z = (z_F)_F$ of 
 $\bidual^r_\Lambda U_{L_\infty, \chi}$ for which one has 
\begin{equation} \label{zeta-element-1}
\{
f (z) \mid f \in \exprod^r_\Lambda P^\ast \} =
\Fitt^r_{\Lambda} ( H^2 (C_{L_\infty, \chi})) =
\Fitt^0_\Lambda \big ( \Cl (L_\infty)_\chi \big)
\end{equation}
and
\begin{equation} \label{zeta-element}
\{ f(z_F) \mid f \in  \exprod^r_{\Z_p [\cG_F]_\chi} P_F^\ast \} = \Fitt^r_{\Z_p [\cG_F]_\chi} ( H^2 (C_{F, \chi})) = \Fitt^0_{\Z_p [\cG_F]} ( \Cl (F))_\chi.
\end{equation}
Upon combining these descriptions with (\ref{limit-comp}), we can therefore derive the displayed equality in  Theorem \ref{classical-imc-result}\,(i) as a direct consequence of (\ref{bss-limit-step-1}). \medskip \\
To prove Theorem \ref{classical-imc-result}\,(ii), we set $\mathcal{E} :=  \Lambda \cdot \eta^{b, \chi}_{L_\infty / K, S}$ and claim that $\mathcal{E}$ is a free $\Lambda$-module of rank one. It is sufficient to prove that $\Ann_\Lambda ( \eta^{b, \chi}_{L_\infty / K, S}) = 0$, which will follow if $\Fitt^0_\Lambda (\Cl (L_\infty)_\chi) \subseteq \im (\eta^{b, \chi}_{L_\infty / K, S})$ contains a non-zero divisor. Observe that 
\begin{equation} \label{fitting-ideal-span}
Q (\Lambda) \cdot \Fitt^0_\Lambda ( \Cl (L_\infty)_\chi) = \Fitt^0_{Q (\Lambda)} ( Q (\Lambda) \cdot \Cl (L_\infty)_\chi) = Q (\Lambda),
\end{equation}
where the first equality follows from a natural property of Fitting ideals and the second from the fact that $\Cl (L_\infty)_\chi$ is a $\Lambda$-torsion module, and so $Q (\Lambda) \cdot \Cl (L_\infty)_\chi= 0$. 
Thus, $\Fitt^0_\Lambda ( \Cl (L_\infty)_\chi)$ spans $Q (\Lambda)$ over $Q (\Lambda)$ and, in particular, must contain a non-zero divisor. \\
Moreover, the representative (\ref{quad rep}) of $C_{\infty, \chi}$ can be used to compute that the $Q(\Lambda)$-module $Q (\Lambda) \cdot U_{L_\infty, \chi}$ is free of rank $r$. It follows that $\big ( \bidual^r_\Lambda U_{L_\infty, \chi} \big ) / \mathcal{E}$ is the quotient of two modules that each span a free module of rank one over $Q(\Lambda)$ and hence must be $\Lambda$-torsion. 
Consequently, dualising the tautological exact sequence
\[
0 \longrightarrow  \mathcal{E}
\longrightarrow \bidual^r_\Lambda U_{L_\infty, \chi}
\longrightarrow  \big ( \bidual^r_\Lambda U_{L_\infty, \chi} \big ) / \mathcal{E}
\longrightarrow 0
\]
gives the  exact sequence
\begin{equation*} \label{transpose-3}
0 \to \big ( \displaystyle \bidual^r_\Lambda U_{L_\infty, \chi} \big )^{\ast}
\to  \mathcal{E}^\ast
\to  \EE^1_\Lambda \Big ( \big ( \bidual^r_\Lambda U_{L_\infty, \chi} \big ) / \mathcal{E}  \Big)
\stackrel{\kappa}{\to} \EE^1_\Lambda \big (  \bidual^r_\Lambda U_{L_\infty, \chi} \big )
\to 0.
\end{equation*}
Then, since the $\Lambda$-module $\mathcal{E}^\ast$ is free of rank one, the latter exact 
 sequence can be used to calculate that $ \Fitt^0_\Lambda (\ker (\kappa))$ coincides with the ideal of $\Lambda$ given by 
 \begin{align*}
 I & : = \im \Big (  \big ( \displaystyle \bidual^r_\Lambda U_{L_\infty, \chi} \big )^{\ast} \to \mathcal{E}^\ast \cong \Lambda \Big ) =\big \{ f (\eta^{b, \chi}_{L_\infty / K, S}) \mid f \in \big ( \displaystyle \bidual^r_\Lambda U_{L_\infty, \chi} \big )^{\ast} \big\}.
 \end{align*}
 The ideal $I$ is the isomorphic image of the reflexive module $ \big ( \displaystyle \bidual^r_\Lambda U_{L_\infty, \chi} \big )^{\ast}$ and therefore reflexive as well. In particular, it is uniquely determined by its localisations at primes of height at most one of $\Lambda$ 
(see, for example, \cite[Lem.\@ C.13]{Sakamoto20}). 
 To establish the first part of claim (ii) it now suffices to prove that $I_{\wp} = \im ( \eta^{b, \chi}_{L_\infty / K, S})_\wp$ for every prime ideal $\wp \subseteq \Lambda$ of height at most one. This follows from the fact that the cokernel of the natural map
 \[
 \exprod^r_\Lambda U_{L_\infty, \chi}^\ast \to  \big ( \exprod^r_\Lambda U_{L_\infty, \chi}^\ast \big )^{\ast \ast}  = \big ( \displaystyle \bidual^r_\Lambda U_{L_\infty, \chi} \big )^{\ast}
 \]
identifies with $\EE^2_\Lambda (X)$ for a certain $\Lambda$-module $X$ (see \cite[Prop.\@ (5.4.9)\,(iii)]{NSW}) and therefore, since $\Lambda$ is Gorenstein, vanishes after localising at $\wp$.
 \medskip \\
Having proved the first equality in claim (ii), the final assertion will follow if we can show that 
   $\EE^1_\Lambda ( \bidual^r_\Lambda U_{L_\infty, \chi})$ vanishes after localising at any prime $\wp \subseteq \Lambda$ of height at most two. However, this follows from the general fact that for any finitely generated $\Lambda$-module $X$, and any such prime ideal $\wp$, one has $\EE^1_\Lambda(X^\ast)_\wp = 0$. To see this one only needs to note that the $\Lambda$-linear dual of any fixed projective presentation
\[
\Pi_1 \stackrel{f}{\longrightarrow} \Pi_0 \longrightarrow  X \longrightarrow 0
\]
of $X$ gives an exact sequence that induces an isomorphism between $\EE^1_\Lambda ( X^\ast)$ and 
 $\EE^3_\Lambda (\coker f^\ast)$, 
 and the latter group vanishes after localisation at $\wp$ since $\Lambda$ is Gorenstein. \medskip\\
This concludes the proof of Theorem \ref{classical-imc-result}. \qed

\subsection{Reformulation in terms of determinants}

In this subsection we explain how one can deduce new evidence towards the validity of 
Conjecture \ref{IMCgeneral} from Theorem \ref{classical-imc-result}. We remark that the idempotent $e_\chi$ associated to a character $\chi \: \Delta \to \overline{\Q_p}^\times$ will satisfy Hypothesis \ref{hypsoule}\,(i)\,(b) since $p$ is odd and, if $\chi$ validates \ref{bss-hypothesis}\,(i), then it will also satisfy Hypothesis \ref{hypsoule}\,(i)\,(c). 

\begin{theorem} \label{det-imc-result}
Assume that the Rubin-Stark Conjecture holds for all finite abelian extensions of $K$ and that the character $\chi$ validates Hypothesis \ref{bss-hypothesis}. If we set
\[
\mathfrak{A} = \{ x \in Q(\Lambda) \mid x \cdot \Fitt^0_\Lambda ( \Cl (L_\infty)_\chi) \subseteq 
\Fitt^0_\Lambda ( \Cl (L_\infty)_\chi) \},
\]
then all of the following claims are valid:
\begin{enumerate}[label=(\roman*)]
\item $\mathfrak{A}$ is a $\Lambda$-order inside $Q (\Lambda)$,
\item $\mathfrak{A}$ is contained in $\Lambda [\frac1p]$ and coincides with $\Lambda$ if either $G$ has no $p$-torsion or the $\mu$-invariant of $\Cl (L_\infty)_\chi$ vanishes,
\item the higher-rank equivariant Iwasawa Main Conjecture for $\GG_m$  (Conjecture \ref{IMCgeneral}) holds for the $\Lambda$-order $\mathfrak{A}$.
\end{enumerate}
\end{theorem}

\begin{proof} 
To prove claim (i), we first recall that $\Fitt^0_\Lambda ( \Cl (L_\infty)_\chi)$ spans $Q (\Lambda)$ over $Q (\Lambda)$, see (\ref{fitting-ideal-span}). It follows that $\mathfrak{A}$ spans $Q (\Lambda)$ and that the ideal $\Fitt^0_\Lambda ( \Cl (L_\infty)_\chi)$ must contain a non-zero divisor $y$, hence $\mathfrak{A} \subseteq y^{-1} \Fitt^0_\Lambda ( \Cl (L_\infty)_\chi)$ is finitely generated over $\Lambda$. \medskip \\
To prove claim (ii), we note that if $x$ belongs to $\mathfrak{A}$,
then there is a non-zero divisor $y \in \Lambda$ such that $f := yx$ is an element of $\Lambda$
and we claim that $f \in y \Lambda [\frac 1 p]$.
The latter is an inclusion of invertible $\Lambda [\frac{1}{p}]$-modules and can therefore be checked locally at height-one primes $\p$ of $\Lambda [\frac1p]$ (see \cite[Lemma~5.3]{Flach-survey}). 
Each such height-one prime $\p$ can be identified with a height-one prime of $\Lambda$ that does not contain $p$. The localisation $\Lambda_\p$ is then a discrete valuation ring (cf.\@ \cite[\S 3C1]{bks2}) and so the ideal $\Fitt^0_\Lambda (\Cl (L_\infty)_\chi)_\p$ is principal and generated by a non-zero divisor $a$, say. It follows that $fa \Lambda_\p \subseteq y a \Lambda_\p$, hence $f \Lambda_\p \subseteq y \Lambda_\p$ as claimed. The same proof also shows the second assertion of claim (ii) since, assuming either of the stated conditions, the ideal $\Fitt^0_\Lambda (\Cl (L_\infty)_\chi)_\p$ is also principal for every height-one prime $\p$ of $\Lambda$ that contains $p$, as required to prove claim (ii).  \medskip \\
To prove claim (iii) we fix a $\Lambda$-basis $\fz \in \det^{-1}_\Lambda (C_{\infty, \chi})$ and set $z : = \pi_\infty (\fz)$. Due to the injectivity of $\pi_\infty$ it suffices to show that, over $\mathfrak{A}$, the elements $z$ and $\eta^{b, \chi}_{L_\infty / K, S}$ generate the same submodule of $\bidual^r_\Lambda U_{L_\infty, \chi}$. We have already observed in the proof of Theorem \ref{classical-imc-result} that $\eta^{b, \chi}_{L_\infty / K, S}$ generates a free module of rank one over $\Lambda$. 
Thus, both $z$ and $\eta^{b, \chi}_{L_\infty / K, S}$ span $\bidual^r_\Lambda U_{L_\infty, \chi}$ over $Q (\Lambda)$. Consequently, there is $x \in Q (\Lambda)$ such that $z = x \cdot \eta^{b, \chi}_{L_\infty / K, S}$ and we need to show that $x \in \mathfrak{A}^\times$. Notice that
\[
\{ f (\eta^{b, \chi}_{L_\infty / K, S} )  \mid f \in \exprod^r_\Lambda P^\ast \} \cdot x = \{ f (z) \mid f \in \exprod^r_\Lambda P^\ast \},
\]
hence claim (iii) follows from the observations (\ref{images}), (\ref{limit-comp}),  (\ref{zeta-element-1}) and (\ref{bss-limit-step-1})
made in the proof of Theorem \ref{classical-imc-result}.  
\end{proof}

\section{Evidence for the Soul\'e-Stark Conjecture}\label{ss section}

In this section we explain the relation of the Soul\'e-Stark Conjecture (Conjecture \ref{conj}) to a variety of results and conjectures in the literature. This gives a better understanding of Conjecture \ref{conj} and allows us to  
interpret several important existing results as supporting evidence for both the Soul\'e-Stark Conjecture and the finer Congruence Conjecture (Conjecture \ref{cc conj}). In the setting of CM abelian varieties we also find that this  approach clarifies aspects of the work of B\"uy\"ukboduk and Lei in \cite{BL1} and \cite{BL2} and combines with the results of Theorems \ref{descent} and \ref{det-imc-result} to give new evidence in support of the equivariant Tamagawa Number Conjecture.

\subsection{Totally real and CM fields}\label{tr cm section}

In this section we fix a totally real field $K$ and a finite abelian CM extension $L$ of $K$ with Galois group $G$.

 We consider the motive $(h^0(L)(j)/K,e_j^\pm \QQ[G])$ from Example \ref{exrank}(ii) and take the idempotent $\varepsilon$ in Hypotheses \ref{hypcong} and \ref{hypsoule} to be a suitable choice of $e^\pm := \frac{1\pm c}{2}$.

We remark that the conjecture of Solomon \cite{solomon} that occurs in claim (iv) of the following result is an explicit reciprocity law for Rubin-Stark elements that extends (conjecturally) the classical explicit reciprocity laws of Artin-Hasse and Iwasawa \cite{iwasawa}.

\begin{proposition}\label{evi}\
\begin{enumerate}[label=(\roman*)]
\item When $j\leq 0$, Conjecture ${\rm SS}(h^0(L)(j)/K,e_j^- \QQ[G])$ is valid.
\item When $j>1$, Conjecture ${\rm SS}(h^0(L)(j)/K,e_j^- \QQ[G])$ is equivalent to the $p$-adic Beilinson conjecture of Besser-Buckingham-de Jeu-Roblot \cite[Conj.\@ 3.18(3)]{BBJR}.
\item When $K=\QQ$, Conjecture ${\rm SS}(h^0(L)(j)/K,e_j^+ \QQ[G])$ is valid.
\item Conjecture ${\rm CC}(h^0(L)(1)/K,e^- \QQ[G],L,n)$ implies Solomon's Congruence Conjecture \cite[{${\rm CC}(L/K, S, p, n -1)$}]{solomon}.
\end{enumerate}
\end{proposition}

\begin{proof} Much of this result is essentially known and so, for brevity, we shall omit many details. 

In this way, we simply note that claim (i) follows from Example \ref{exsoule}\,(ii) and the well-known interpolation property of the Deligne-Ribet $p$-adic $L$-function, that claim (ii) follows from  the observation in \cite[Rem.\@ 3.10]{sbA2} and that claim (iii) follows from Example \ref{exsoule}\,(i) and the results of Beilinson-Huber-Wildeshaus (see \cite[Thm.\@ 5.2.1 and 5.2.2]{HK}) for $j\leq 0$ and Kato (see \cite[Thm.\@ 3.2.6]{HK}) for $j>0$.

Finally, we note that claim (iv) is proved in the following way. We set $r:=[K:\QQ]$. Note that Conjecture ${\rm CC}(h^0(L)(1)/K,e^- \QQ[G],L,n)$ predicts the equality
\begin{eqnarray}\label{solomonversion}
{\rm tw}_{1,n}^r (\eta_S^w(e^+\cdot \ZZ_{p,L/K})) = \overline{\eta_S^b(e^-\cdot \ZZ_p(1)_{L/K}) }
\end{eqnarray}
in $e^-{\bigcap}_{\ZZ/p^n[G]}^r H^1(\cO_{L,S}, \ZZ/p^n)$, where
$${\rm tw}_{1,n}^r \: e^+{\bigcap}_{\ZZ_p[G]}^r H^1(\cO_{L,S}, \ZZ_p(1)) \to e^-{\bigcap}_{\ZZ/p^n[G]}^r H^1(\cO_{L,S}, \ZZ/p^n)$$
is induced by the cyclotomic character $\chi_{\rm cyc} \: G\to {\rm Aut}(\mu_{p^n})\cong (\ZZ/p^n)^\times$. We set $U_{L_p}:=\varprojlim_n (\cO_L \otimes_\ZZ \ZZ_p)^\times/p^n $ and let $\Gamma_{L,S}$ be the Galois group of the maximal abelian pro-$p$ extension of $L$ unramified outside $S$. Let ${\rm rec}_p \: U_{L_p} \to \Gamma_{L,S}$ be the product of local reciprocity maps at primes above $p$. It induces a map
$${\rm rec}_p^\ast \: e^- {\bigcap}_{\ZZ_p[G]}^r H^1(\cO_{L,S},\ZZ_p) = e^- {\bigcap}_{\ZZ_p[G]}^r \Hom_{\rm cont}(\Gamma_{L,S},\ZZ_p) \to e^- {\bigcap}_{\ZZ_p[G]}^r \Hom_{\ZZ_p}(U_{L_p}, \ZZ_p).$$
The natural modulo $p^n$ version is also denoted by the same symbol.
Note that we have a canonical isomorphism $\Hom_{\ZZ_p}(X,\ZZ_p) \cong \Hom_{\ZZ_p[G]}(X,\ZZ_p[G])$ for any $\ZZ_p[G]$-module $X$. 
Hence by the definition of exterior power biduals we have an identification
$$e^- {\bigcap}_{\ZZ_p[G]}^r \Hom_{\ZZ_p}(U_{L_p}, \ZZ_p) = \Hom_{\ZZ_p[G]e^-} \left( {\bigwedge}_{\ZZ_p[G]e^-}^r e^- U_{L_p}, e^-\ZZ_p[G]\right).$$
Solomon's conjecture predicts an equality of homomorphisms $ {\bigwedge}_{\ZZ_p[G]e^-}^r e^- U_{L_p}\to e^-\ZZ/p^n[G]$, so it is an equality in $\Hom_{\ZZ_p[G]e^-} \left( {\bigwedge}_{\ZZ_p[G]e^-}^r e^- U_{L_p}, e^-\ZZ/p^n[G]\right)$.
One checks that it is equivalent to the equality
$${\rm rec}_p^\ast\left( {\rm tw}_{1,n}^r (\eta_S^w(e^+ \cdot\ZZ_{p,L/K}))\right) ={\rm rec}_p^\ast\left( \overline{\eta_S^b(e^-\cdot \ZZ_p(1)_{L/K})} \right),$$
which is obviously implied by (\ref{solomonversion}).
\end{proof}

\begin{remark} The basic rank is equal to zero in the cases of Proposition \ref{evi}\,(i) and (ii) and equal to one in the case of Proposition \ref{evi}\,(iii). In the case of Proposition \ref{evi}(iv), the basic rank is equal to $[K:\QQ]$, and so it gives evidence for the Congruence Conjecture in the higher basic rank case (see, for example, the extensive evidence in support of Solomon's Conjecture obtained by Roblot and Solomon in \cite{rs}).
\end{remark}

\subsection{CM abelian varieties and Hecke characters} \label{BL-section}

The Soul\'e-Stark Conjecture for CM elliptic curves and, more generally, for Hecke characters (as in Example \ref{exrankone}\,(ii)) has essentially been studied by many authors including Coates-Wiles \cite{CW}, Kato \cite{katolecture}, Kings \cite{kings}, Tsuji \cite{tsuji} and Bars \cite{bars}. In particular, the questions of Tsuji in \cite[\S 11]{tsuji} explicitly describe Conjecture ${\rm SS}(M/K,R)$ in this case.

In this section we shall apply our approach in the general setting of CM abelian varieties (as outlined in Example \ref{exrankone}\,(iii)), and thereby shed new light on work of B\"uy\"ukboduk and Lei in \cite{BL1} and \cite{BL2}.

In the remainder of this section we will therefore assume that $K$ is a CM field and we set 
\[ g := \tfrac12 [K : \Q].\]

\subsubsection{The Perrin-Riou-Stark Conjecture} \label{Perrin-Riou-Stark-section}

We first make an observation concerning  a conjecture of B\"uy\"ukboduk and Lei. 

To state the conjecture we fix a character $\chi \: G_K \to \overline{\Q_p}^\times$ of the absolute Galois group $G_K$ of $K$ that has finite prime-to-$p$ order, is not  equal to the Teichm\"uller character $\omega$ and is such that $\chi (\mathfrak{p}) \neq 1$ for every $p$-adic place $\mathfrak{p}$ of $K$. 

We write $L$ for the field cut out by the character $\chi$, take $L_\infty$ to be the composite of $L$ with the maximal $\Z_p$-power extension of $K$, and set 
$\Lambda := \Z_p [\im \chi] \llbracket \Gal (L_\infty/L) \rrbracket$. Let $S$ be the finite set of places $S_\infty (K) \cup S_\mathrm{ram} (L / K) \cup S_p (K)$. As before, we use the notation 
\[ U_{L_\infty, \chi} := \varprojlim_F ( \Z_p \cdot \mathcal{O}_{F, S}^\times )_\chi,\]
where $F$ ranges over all finite extensions in $L_\infty / L$. We also fix a $\Lambda$-basis $b$ of the module  $Y_K(\ZZ_{p,L_\infty/K})_\chi$.

\begin{conjecture}[B\"uy\"ukboduk-Lei] \label{perrin-riou-stark}
There exists a subset $\{\mathfrak{S}^\chi_{\infty,i}\}_{1\le i\le g}$ of $U_{L_\infty, \chi}$ with the property that, for every finite extension $F$ of $L$ in $L_\infty$, the natural map    
\[
 {\bigwedge}^g_\Lambda U_{L_\infty, \chi} \to {\bigwedge}^g_{\Z_p [\im \chi] [\Gal (F / L)]} ( \Z_p \cdot \mathcal{O}_{F, S}^\times)_\chi \to  {\bigcap}^g_{\Z_p [\im \chi] [\Gal (F / L)]} ( \Z_p \cdot \mathcal{O}_{F, S}^\times)_\chi
\]
sends ${\bigwedge}_{i=1}^{i=g}\mathfrak{S}^\chi_{\infty,i}$ to the $\chi$-component $\eta^{b, \chi}_{F / K, S}$ of the relevant Rubin-Stark element.
\end{conjecture}

\begin{remark} The above conjecture is first formulated in \cite[Conj.~4.14]{BL1}, where it is referred to as the \textit{Perrin-Riou-Stark Conjecture}, and also plays an important role in the subsequent articles of 
 B\"uy\"ukboduk and Lei \cite{BL2} and of  B\"uy\"ukboduk \cite{buyukboduk-3} (in the latter of which the conjecture is referred to as the \textit{Strong Rubin-Stark Conjecture}). The conjecture also makes an appearance in \cite[\S 3.4.4]{buyukboduk-4}. \end{remark} 

In connection to the following result, we refer the reader to Remark \ref{BL-errors} below. 

\begin{proposition}\label{bl problems} 
Assuming the hypotheses as stated above, the following claims are valid.
\begin{enumerate}[label=(\roman*)]
\item If $U_{L_\infty, \chi}$ is a free $\Lambda$-module, then the validity of Conjecture \ref{perrin-riou-stark} follows directly from the validity of the Rubin-Stark Conjecture.
\item If $\Cl (L)_\chi$ vanishes, then $U_{L_\infty, \chi}$ is a free $\Lambda$-module.
\item If $\Cl (L)_\chi$ does not vanish and $U_{L_\infty, \chi}$ is a free $\Lambda$-module, then the height of the annihilator of $\Cl (L_\infty)_\chi$ in $\Lambda$ is at most two. 
\end{enumerate}
\end{proposition}

\begin{proof} 
At the outset we note $\chi$ is a faithful character of $\Gal (L / K)$ and therefore
satisfies Hypothesis \ref{bss-hypothesis}\,(iii) as a consequence of our assumption $\chi (\p) \neq 1$ for every $\p$-adic place $\p$ of $K$. 
Since moreover $\chi \not= \omega$, the construction of Theorem \ref{classical-imc-result}\,(i) shows that, for every finite extension $F$ of $L$ in $L_\infty$, the natural corestriction map    
\[
 {\bigcap}^g_\Lambda U_{L_\infty, \chi} \to {\bigcap}^g_{\Z_p [\im \chi] [\Gal (F / L)]} ( \Z_p \cdot \mathcal{O}_{F, S}^\times)_\chi \]
sends the element $\eta^{b,\chi}_{L_\infty/K,S}$ to $\eta^{b, \chi}_{F / K, S}$. 
To prove claim (i) it is therefore enough to show that, if $U_{L_\infty, \chi}$ is a free $\Lambda$-module, in which case the natural map ${\bigwedge}^g_\Lambda U_{L_\infty, \chi} \to {\bigcap}^g_\Lambda U_{L_\infty, \chi}$ is bijective (cf.\@ \cite[Lem.~A.1]{sbA}), then $\eta^{b, \chi}_{F / K, S}$ has the form predicted by Conjecture \ref{perrin-riou-stark}.\\
To prove this it is in turn enough to show that $Q(\Lambda) \otimes_\Lambda U_{L_\infty, \chi}$ is a free $Q(\Lambda)$-module of rank $g$ since then $U_{L_\infty, \chi}$ must (if free) be a free $\Lambda$-module of rank $g$ and so every element of ${\bigwedge}^g_\Lambda U_{L_\infty, \chi}$ is of the form ${\bigwedge}_{i=1}^{i=g}u_i$ for a suitable subset $\{u_i\}_{1\le i\le g}$ of $U_{L_\infty, \chi}$.\\
The key point now is that the $\Lambda$-Euler characteristic of the complex $C_{L_\infty, \chi}$ used in the proof of Theorem \ref{classical-imc-result} vanishes (as follows directly from the representative (\ref{quad rep})). Indeed, since $\Cl (L_\infty)_\chi$ is a torsion $\Lambda$-module and $Q(\Lambda) \otimes_\Lambda Y_K (\Z_{p, L_\infty / K})_\chi$ is a free $Q(\Lambda)$-module of rank $g$, this observation combines with the explicit descriptions of cohomology given in (\ref{can iso 1}) and (\ref{can iso 2}) and the fact that $Q(\Lambda)$ is a finite product of fields to imply that the $Q(\Lambda)$-module $Q(\Lambda) \otimes_\Lambda U_{L_\infty, \chi}$ is also free of rank $g$, as required.\medskip\\
To prove (ii) we observe that the natural projection isomorphism $C_{L_\infty, \chi} \otimes^\mathbb{L}_\Lambda \Z_p [\Gal (L / K)] \cong C_{L, \chi}$ induces an isomorphism of $\ZZ_p[\Gal(L/K)]$-modules $H^2 (  C_{L_\infty, \chi}) \otimes_\Lambda \Z_p [\Gal (L / K)] \cong H^2 ( C_{L, \chi})$. Then, since the $\Lambda$-module $Y_K ( \Z_{p, L_\infty / K})_\chi$ is free, the exact sequence (\ref{can iso 2}) implies that the latter isomorphism restricts to give an isomorphism $\Cl (L_\infty)_\chi \otimes_\Lambda \Z_p [\Gal (L / K)] \cong \Cl (L)_\chi$. Given this isomorphism, Nakayama's Lemma implies $\Cl (L_\infty)_\chi$  vanishes if $\Cl (L)_\chi$ vanishes. In this case, therefore, the cokernel of the endomorphism $\theta$ in (\ref{quad rep}) is $Y_K ( \Z_{p, L_\infty / K})_\chi$ and so, in particular, is $\Lambda$-free. By splitting up the four-term exact sequence associated to (\ref{quad rep}) into short exact sequences, we can therefore deduce that $\ker (\theta) \cong U_{L_\infty, \chi}$ is $\Lambda$-projective, and hence free, as required. \medskip \\
To prove claim (iii) we assume $U_{L_\infty, \chi}$ is a free $\Lambda$-module. Then, since the $\Lambda$-module $Y_K ( \Z_{p, L_\infty / K})_\chi$ is also free, the four-term exact sequence associated to (\ref{quad rep}) can in this case be  
adjusted to give a projective resolution of $\Cl (L_\infty)_\chi$ which shows that the projective dimension of the $\Lambda$-module $\Cl (L_\infty)_\chi$ is at most two. 
If we now assume that $\Cl (L)_\chi = \Cl (L_\infty)_\chi \otimes_\Lambda \Z_p [\Gal (L / K)]$ is non-zero, then also $\Cl (L_\infty)_\chi$ is non-zero. 
The assertion of claim (iii) is therefore true because any minimal associated prime $\mathfrak{q}$ of the non-zero module $\Cl (L_\infty)_\chi$ belongs to $\mathrm{Supp} ( \Cl (L_\infty)_\chi)$ and satisfies 
\[ {\rm height}(\mathfrak{q}) = d+1-\dim ( \Lambda / \mathfrak{q}) \leq 
d+1-\mathrm{depth} ( \Cl (L_\infty)_\chi) = \mathrm{pd} ( \Cl (L_\infty)_\chi) \le 2.
\]
Here we write $d$ for the rank of $\Gal(L_\infty/L)$ so that the first equality is obvious, the first inequality is true because $\dim ( \Lambda / \mathfrak{q}) \geq \mathrm{depth} ( \Cl (L_\infty)_\chi) $ (cf.\@ \cite[Thm.\@ 17.2]{matsumura}) and
the second equality follows directly from the Auslander-Buchsbaum formula. 
\end{proof}

\begin{remark}\label{BL-errors} 
The result of \cite[Lem.~4.11]{BL1} asserts that $U_{L_\infty, \chi}$ is a free $\Lambda$-module. Unfortunately, however, there is an error in the argument given in \textit{loc.\@ cit.\@} (and alluded to in \cite[Rem.~2.16]{buyukboduk-3}) that is yet to be fixed. In this context, the point of Proposition \ref{bl problems}\,(i) is that whenever $U_{L_\infty, \chi}$ is a  free $\Lambda$-module the validity of the Perrin-Riou-Stark Conjecture follows directly from other standing assumptions that are made in each of the articles \cite{BL1}, \cite{BL2} and \cite{buyukboduk-3}. In general, Proposition \ref{bl problems}\,(iii) shows that, if ${\rm Cl}(L)_\chi$ does not vanish, then the freeness of $U_{L_\infty, \chi}$ imposes a strong bound on the height of the annihilator over $\Lambda$ of $\Cl (L_\infty)_\chi$ (thereby showing that the latter module cannot be `too small'). For comparison, we note that a result of Sharifi \cite[Cor.~4.3]{sharifi} implies the annihilator of the class group $\Cl (L^{\rm max}_\infty)$ over the Iwasawa algebra associated to the maximal $\Z_p$-power extension $L^{\rm max}_\infty$ of $L$ can have arbitrarily large height.
 \end{remark}

\subsubsection{The explicit reciprocity conjecture of B\"uy\"ukboduk and Lei}\label{exp rec BL}

To review this conjecture we fix a principally polarised abelian variety $A$ over $K$ that has complex multiplication by $K$. We assume that $K$ contains the reflex field of $A$ and that the index of $\text{End} (A)$ inside the maximal order $\mathcal{O}$ of $K$ is coprime to $p$.  We note that the latter assumption implies that the Tate module $T_p A = \varprojlim_n A [p^n]$ of $A$ is free of rank one over the $\Z_p$-order $\bigO_p  := \Z_p\otimes_\Z \bigO$ (see \cite[Rem., p.~502]{SerreTate}). 
For each finite Galois extension $L$ of $K$, and each $p$-adic prime $\p$ of $L$, we write\begin{align*}
\res_\p \: H^1 ( \bigO_{K, S}, (V_p A)_{L/ K}) \cong H^1 ( \bigO_{L, S}, V_p A)  \to H^1 ( L_\p, V_p A ),
\end{align*}
for the natural localisation map. Furthermore, we write
\[
\exp^\ast_{\p} \:  H^1 ( L_\p, V_p A ) \to H^1_{/ f} ( L_\p, V_p A )
\stackrel{\cong}{\longrightarrow} \mathrm{Fil}^0 D_{\mathrm{dR}, L_\p} (V_p A)  
\]
for the associated dual exponential map. We shall use the associated homomorphism of $K_p[\Gal(L/K)]$-modules 
\[ {\rm exp}^\ast_{L,p} \: H^1 ( \bigO_{K, S}, (V_p A)_{L/ K}) \xrightarrow{ (\exp^\ast_{\p}\circ\res_\p)_\p} 
\mathrm{Fil}^0 D_{\dR, L_p} (V_p A), 
\]
where we have set
$D_{\dR, L_p} (V_p A) = \bigoplus_{\p\mid p}   D_{\mathrm{dR}, L_\p} (V_p A)$ and
in the direct sum $\p$ runs over all $p$-adic places of $L$ above $\p$.  \medskip \\
We also recall that, if $\sha ( A / L) [p^\infty]$ is finite, then Poitou-Tate duality implies that there is a canonical isomorphism  
\begin{equation}\label{strict selmer} 
H^2 (\bigO_{L, S}, V_pA) \cong \ker \Big (
\Q_p \otimes_\Z A (L) \stackrel{\lambda}{\longrightarrow} 
\bigoplus_{\p \mid p} ( \Q_p \otimes_{\Z_p} A(L_\p)^{\wedge} )\Big )^\ast,
\end{equation}
where in the direct sum $\p$ runs over all $p$-adic places of $L$, $( - )^\wedge$ denotes $p$-adic completion and $\lambda$  denotes the natural diagonal localization map (see, for example, the argument in \cite[Lem.~6.1\,(ii)]{bss2}). 

We now write $K_\infty^\cyc / K$ for the cyclotomic $\Z_p$-extension of $K$. For each natural number $n$ we write $K^\cyc_n$
 for the $n$-th layer of $K_\infty^\cyc / K$ and set $\Gamma^\cyc_n := \Gal (K^\cyc_n / K)$. 

Then the explicit reciprocity conjecture of B\"uy\"ukboduk and Lei predicts that for every $n$ there exists a subset $\{ c_{n,i} \}_{1 \leq i \leq g}$ of $H^1 ( \bigO_{K, S}, (T_p A)_{K_n^\cyc / K})$ with both of the following properties. Each element $c_{n,i}$ is obtained by `twisting' the element $\bigwedge_{1 \leq i \leq g} \mathfrak{S}^\chi_{\infty, i}$ predicted to exist by the Perrin-Riou-Stark Conjecture \ref{perrin-riou-stark} (for a suitable choice of $\chi$) by the character $G_K \to \mathcal{O}_p[\Gamma_n]^\times$ induced by the action of $G_K$ on $(T_p A)^\ast (1)_{K_n^\cyc / K}$. In addition, if one fixes an embedding $\overline{\Q_p} \to \CC$ and uses it to identify the groups $\Hom(\Gamma_n,\overline{\Q_p}^\times)$ and $\Hom(\Gamma_n,\CC^\times)$, then  for every primitive character $\theta\: \Gamma_n \to \overline{\Q_p}^\times$ there is an equality in $\CC_p$ of the general form  
\begin{equation}\label{BL conj} 
{\rm det}\bigl(\bigl[ e_\theta \cdot {\rm exp}^\ast_{K_n^\cyc,p}( c_i), \kappa_{\theta,j}\bigr]_{A,n}\bigr)_{1 \leq i,j \leq g} = L_{\{p\}}(A,\theta^{-1},1)\cdot \Omega_{A,n,p}.
\end{equation}
Here $[-,-]_{A,n}$ denotes a canonical $\CC_p$-bilinear pairing on Dieudonn\'e modules, the elements $\kappa_{\theta,j}$ are obtained from a suitable choice of basis of the relevant Dieudonn\'e module, $L_{\{p\}}(A, \theta^{-1}, s)$ is the $p$-truncated Hasse-Weil-Artin $L$-series attached to $A$ and $\theta^{-1}$ and $\Omega_{A,n,p}$ is the product of the ratio of the canonical $p$-adic and complex periods associated to the pair $(A,\theta)$ with a fudge factor that compensates for the precise choices of elements $\{\mathfrak{S}^\chi_{\infty, i}\}_{1\le i\le g}$ and $\{\kappa_{\theta,j}\}_{1\le j\le g}$. 

Full details concerning the conjecture of B\"uy\"ukboduk and Lei and a precise version of the conjectural equality (\ref{BL conj}) can be found in \cite{BL1}. However, the partial details recalled above are at least sufficient to prove the conjecture is of interest only in the case that the space $e_\theta \big ( \overline{\Q_p} \otimes_{\Z} A ( K_n^\cyc) \big )$ vanishes. This fact is shown by the following result, in which we refer to the Deligne-Gross Conjecture for Hasse-Weil-Artin $L$-series (for a precise statement of which see, for example, \cite[p.\,127]{rohrlich}). 

\begin{proposition} \label{BL-lemma} Assume that the Perrin-Riou-Stark Conjecture (Conjecture \ref{perrin-riou-stark}) is valid and that the Hasse-Weil-Artin $L$-series $L(A,\theta^{-1},s)$ validates the Deligne-Gross Conjecture. 

Then, if $\sha ( A / K_n^\cyc) [p^\infty]$ is finite and $e_\theta \big ( \overline{\Q_p} \otimes_{\Z} A ( K_n^\cyc) \big )$ is non-zero, the explicit reciprocity conjecture of B\"uy\"ukboduk and Lei \cite[Conj.~4.18]{BL1} is valid.  
\end{proposition}

\begin{proof} We set $L := K_n^\cyc$, $W_\theta := e_\theta(\overline{\Q_p} \otimes_{\Z} A (L))$ and  $W_\theta' := e_\theta(\overline{\Q_p} \otimes_{\Q_p} \ker(\lambda)^\ast)$ where the map $\lambda$ is as in (\ref{strict selmer}). 
 We also write $\psi$ for the homomorphism $\Gamma_n \to \overline{\QQ}^\times$ obtained by composing $\theta$ with a fixed choice of embedding $\overline{\Q}_p \to \CC$. 

Then the non-vanishing of $W_\theta$ implies that $e_\psi(\overline{\Q}\otimes_\Z A(L))$
 does not vanish and this fact combines with the assumed validity of the Deligne-Gross Conjecture for $L(A,\psi^{-1},s)$ to directly imply that $L(A,\psi^{-1},1)$, and hence also $L_{\{p\}}(A,\psi^{-1},1)$, vanishes.  

To prove the second assertion it is therefore enough to show that the stated assumptions imply the determinant that occurs on the left hand side of the conjectured equality (\ref{BL conj}) vanishes. 

As a first step, we claim that 
\begin{equation}\label{first step} {\rm dim}_{\overline{\Q_p}}(W_\theta) > {\rm dim}_{\overline{\Q_p}}(W_\theta').\end{equation}
To verify this we note that the non-degeneracy of the N\'eron-Tate height pairing implies the $\overline{\QQ_p}[G]$-modules $\overline{\QQ_p} \otimes_\ZZ A(L)$ and $\overline{\QQ_p}\otimes_{\Z}\Hom_\Z(A(L),\Z)$ are (non-canonically) isomorphic and hence that ${\rm dim}_{\overline{\Q_p}}(W_\theta) = {\rm dim}_{\overline{\Q_p}}(W_{\theta^{-1}})$. In addition, if we set $W^\dagger_\theta := e_{\theta^{-1}}(\overline{\Q_p} \otimes_{\Q_p} \ker(\lambda))$, then one has ${\rm dim}_{\overline{\Q_p}}(W_\theta') = {\rm dim}_{\overline{\Q_p}}(W^\dagger_\theta)$. To justify (\ref{first step}) it is therefore enough to show that $W_\theta^\dagger$ is a proper subspace of $W_{\theta^{-1}}$. 

To do this we write $\tilde e$ for the primitive idempotent $\sum_{\psi'}e_{\psi'}$ of $\Q[G]$, where the sum runs over the set of $G_{\Q}$-conjugates of $\psi^{-1}$. Then the assumed non-vanishing of $W_\theta$, and hence also of $W_{\theta^{-1}}$, implies that $e_{\psi^{-1}}\bigl(\overline{\Q}\otimes_\Z A(L))$, as well as the subspace $\widetilde{W}:= \tilde e\bigl(\Q\otimes_\Z A(L))$ of $\Q\otimes_\Z A(L)$, does not vanish. But then, for any non-zero element $w$ of $\widetilde{W}$, the element 
$\lambda(w)$ is non-zero and so spans a free $\Q[G]\tilde e$-submodule of $\bigoplus_{\p \mid p} ( \Q_p \otimes_{\Z_p} A(L_\p)^{\wedge} )$. The element $e_{\psi^{-1}}(\lambda(w))$ is therefore non-zero and so $e_{\psi^{-1}}(1\otimes w)$ is an element of $W_{\theta^{-1}}$ that does not belong to $W_\theta^\dagger$, as required. 

We now set $Z_\theta := e_\theta( \overline{\Q_p} \otimes_{\Q_p} H^1 (\bigO_{K, S}, (V_p A)_{L/ K} ))$ and write 
$\kappa_\theta$ for the natural (injective) Kummer map $W_\theta \to Z_\theta.$ Then one has  

\begin{align*} {\rm dim}_{\overline{\Q_p}}({\rm coker}(\kappa_\theta))=&\, {\rm dim}_{\overline{\Q_p}}(Z_\theta) - {\rm dim}_{\overline{\Q_p}}(W_\theta)\\
<&\, {\rm dim}_{\overline{\Q_p}}(Z_\theta) - {\rm dim}_{\overline{\Q_p}}(W'_\theta)\\ 
 =&\, {\rm dim}_{\overline{\Q_p}}\bigl(e_\theta ( \overline{\Q_p} \otimes_{\Q_p} Y_K ( (V_p A)_{L/ K}))\bigr)\\
 =&\, g\end{align*}
where the first and last equalities are obvious, the inequality follows from (\ref{first step}) and the second equality follows upon combining the vanishing of $H^0 (\bigO_{K, S}, (V_p A)_{L/ K} ) \cong H^0 (\bigO_{L, S}, V_p A)$ with the identification (\ref{strict selmer}) and the general results of Lemma \ref{compact lemma}\,(i) and (ii).    

In particular, since the image of $\kappa_\theta$ belongs to the kernel of  $e_\theta(\overline{\Q_p}\otimes_{\Q_p}(\exp^\ast_{\p} \circ \res_\p))$ for every $p$-adic place $\p$ of $L$, the above inequality implies that the $\overline{\Q_p}$-space $e_\theta\cdot (\overline{\Q_p}\otimes_{\Q_p} \im(\exp^\ast_{L, p}))$ has dimension strictly less than $g$. The elements   
$\{ e_\theta \cdot \exp^\ast_{L, p}(c_i)\}_{1\le i\le g}$ are therefore linearly dependent and so the bilinearity of the pairing $[-,-]_{A,n}$ implies that the determinant on the left hand side of the conjectured equality (\ref{BL conj}) vanishes, as required. 
\end{proof}

\subsubsection{The Soul\'e-Stark Conjecture and explicit reciprocity}

In this section we show that the Soul\'e-Stark Conjecture refines (a reformulated version of) the explicit reciprocity conjecture of B\"uy\"ukboduk and Lei discussed above. We note, in particular, that whilst the formulation of the latter conjecture assumes the validity of Conjecture \ref{perrin-riou-stark}, the approach used here is independent of this conjecture. 
\medskip \\
At the outset we fix an abelian extension $L$ of $K$, set  $G := \Gal ( L / K)$ and consider the motive $M = h^1 ( A / L) (1)$ for an abelian variety $A$ as introduced in \S\ref{exp rec BL}, regarded as defined over $K$ and with coefficients in $K[G]$. To avoid confusion, in the sequel we write $R_0$ instead of $K$ if we consider its complex multiplication action on $A$  and  reserve the notation $K$ for when we mean the field of definition of the motive $M$. We also set $R = R_0 [G]$ and $R_p = \Q_p \otimes_\Q R$.\\
The $p$-adic \'etale realisation $V : = V_p (M)$ of $M$ is then given by the rational Tate module $(V_p A)^\ast (1)_{L / K} = (\Q_p \otimes_{\Z_p} T_p A)^\ast (1)_{L / K}$ of $A$. In addition, since $p$ is assumed to be coprime to the index of $\End (A)$ inside $\mathcal{O}$, the module $(T_p A)_{L / K}$ is free over $\cR_p = \bigO_p [G]$ and so we can choose $T = (T_p A)^\ast (1)_{L / K}$ as a $G_K$-stable lattice inside $V$. \medskip \\
Recall the notation introduced in \S \ref{soule-stark-section}. Fix an $\cR_p$-basis $b = \{b_1, \dots, b_g \}$ of $Y_K (T)$. Since $K$ is totally imaginary, the hypotheses \ref{hypsoule}\,(i)\,(b) and (ii) are automatically satisfied. 
Moreover, we write
\begin{equation} \label{pairing}
[-, -]_\theta \: \CC_p \cdot \det_{R_p} \big ( \mathrm{Fil}^0 D_{\dR, L_p} (V_p A) \big ) 
\times \CC_p \cdot \det_{R_p} \big ( D_{\dR, L_p} ((V_p A)^\ast (1)) / \mathrm{Fil}^0 \big ) \to \CC_p \otimes_\Q K
\end{equation}
for the pairing induced by the composite of the natural pairing
\[
\mathrm{Fil}^0 D_{\dR, L_p} (V_p A) \times D_{\dR, L_p} ((V_p A)^\ast (1)) / \mathrm{Fil}^0 
\to R_p
\]
in $p$-adic Hodge theory and the homomorphism $ R_p \to \Q_p \otimes_\Q K$ induced by a character $\theta \: G \to \overline{\Q_p}^\times$. 

With this notation in place, we can now give a concrete interpretation of Conjecture \ref{conj} as an \textit{explicit reciprocity conjecture} closely related to the conjectural equality (\ref{BL conj}) of B\"uy\"ukboduk and Lei. 

Taking account of Proposition \ref{BL-lemma}, we restrict attention to 
characters $\theta \: G \to \overline{\Q_p}^\times$ such that $e_\theta (\overline{\Q_p} \otimes_{\Z} A (L))$ vanishes. Furthermore, we write
\[
\psi \: G_K \to \cR_p^\times
\]
for the character afforded by the action of $G_K$ on $(T_p A)_{L / K}$. For all $s \in \CC$ such that $\text{Re } (s) > \frac32$, the associated ($S$-truncated) $L$-function is then given by
\[
L_S (\psi, \theta^{-1}, s) = \prod_{v \not \in S}  (1 - (\theta^{-1} \circ \psi) (v) \cdot \mathrm{N} v^{-s} )^{-1} \quad \in \CC_p \otimes_\Q K,
\] 
where by $\theta^{-1}$ we mean the morphism $\CC_p \otimes_\Q R \to \CC_p \otimes_\Q K$ induced by $\theta^{-1}$, and extended to the whole complex plane by analytic continuation.

\begin{proposition} \label{soule-stark-ab-varieties}
Assume $\sha ( A / L) [p^\infty]$ is finite and let $\theta \: G \to \overline{\Q_p}^\times$ be a character for which both  $e_\theta (\overline{\Q_p} \otimes_{\Z} A (L))$ vanishes and $L_S(\psi, \theta^{-1},1)\not= 0$. Fix an $R_p$-basis $\{c_1, \dots, c_g \}$ of $D_{\dR, L_p} ((V_p A)^\ast (1)) / \mathrm{Fil}^0 $ and set $c = \bigwedge_{1 \leq i \leq g} c_i$. \\
Then the $e_\theta$-component of Conjecture ${\rm SS}(M / K, R)$ is valid if and only if in $\CC_p \otimes_\Q K$ one has 
\begin{equation} \label{soule-stark}
 \big [ ({\bigwedge}^g_{R_p}\exp^\ast_{L,p})( \beta^b_S (T)), \, c \big ]_\theta = 
  L_S (\psi, \theta^{-1}, 1) \cdot \frac{\Omega^p(A,\theta)_{c, \omega}}{\Omega(A,\theta)_{b, \omega}}. 
\end{equation}
Here $\Omega^p(A,\theta)_{c, \omega}$ and $\Omega(A,\theta)_{b, \omega}$ are the canonical $p$-adic and complex periods of $A$ and $\theta$, respectively, normalised with respect to the bases $c$ and $\omega$, and $b$ and $\omega$, for any choice of $R$-basis $\omega$ of $H^0 (A, \Omega^1_{A / L})$ (and are defined precisely in the course of the proof below).
\end{proposition}

\begin{proof} The assumed vanishing of $e_\theta (\overline{\Q_p} \otimes_{\Z} A (L))$ combines with the identification  (\ref{strict selmer}) to imply that the space $e_\theta (\overline{\Q_p}\otimes_{\Q_p}H^2 (\bigO_{K, S}, (V_p A)_{L /K}))$ vanishes and hence that $e_\theta$ is a summand of the \textit{idempotent of admissibility} defined in \cite[Def.\@ 4.8]{bss2}. Given this, the $e_\theta$-component of the \textit{period-regulator isomorphism} $\lambda^{\text{BK}}_{M, b, S, L}$ in the definition of the Bloch-Kato element for $M$ as defined in Def.~4.10 of \textit{loc.\@ cit.\@} coincides with the composite map

\begin{align*}
e_\theta \bigl(\CC_p \cdot \exprod^g_{R_p} H^1 (\bigO_{L, S}, V_p A) \bigr)
& \cong e_\theta\bigl( \CC_p \cdot \exprod^g_{R_p} \mathrm{Fil}^0 D_{\dR, L_p} (V_p A)\bigr) \\
& \cong e_\theta \bigl( \CC_p \cdot \exprod^g_{R_p}  \big (  D_{\dR, L_p}( (V_p A)^\ast (1)) / \mathrm{Fil}^0 \big)^\ast\bigr) \\
& \cong e_\theta \bigl( \CC_p \cdot \exprod^g_{R} H^0 (A, \Omega^1_{A / L}) \bigr)\\
& \cong e_\theta \bigl( \CC_p \cdot \exprod^g_{R_p} Y_K (V)^\ast\bigr) \\
& \cong e_\theta \bigl( \CC_p \otimes_\Q R\bigr),
\end{align*}
where the first isomorphism is induced by ${\bigwedge}^g_{R_p}\exp^\ast_{L,p}$, the second by the duality pairing (\ref{pairing}), the third by the comparison isomorphism of $p$-adic Hodge theory, the fourth by the period map, and the last by our fixed choice of basis $b$ for $Y_K (V)$.  \\
Fix an auxiliary $R$-basis $\omega_1, \dots, \omega_g$ of $H^0 (A, \Omega^1_{A / L})$ and set $\omega := \bigwedge_{1 \leq i \leq g} \omega_i$. We then define $\Omega^p(A)_{c, \omega}$ to be the unique element of $R_p$ with the property that the map
\[
  \exprod^g_{R_p}  D_{\dR, L_p} ((V_p A)^\ast (1)) / \mathrm{Fil}^0 
 \stackrel{\cong}{\longrightarrow}
 \Q_p \cdot \exprod^g_R H^0 (A, \Omega^1_{A / L})^\ast
\]
induced by the comparison isomorphism of $p$-adic Hodge theory sends $c$ to $\Omega^p(A)_{c, \omega} \cdot \omega^\ast$.

 Similarly, we define $\Omega(A)_{b, \omega}$ as the unique element of $\CC_p \otimes_\Q R$ such that the map
\[
\CC_p \cdot \exprod_{R_p}^g Y_K (V) \stackrel{\cong}{\longrightarrow}  \CC_p \cdot \exprod^g_R H^0 (A, \Omega^1_{A / L})^\ast
\]
sends $b$ to $\Omega(A)_{b, \omega} \cdot \omega^\ast$.\medskip \\
Then, with these definitions in place, one has 
\begin{align*}
& e_\theta \Bigl(\lambda^{\mathrm{BK}}_{M, b, S, L} \big ( ({\bigwedge}^g_{R_p}\exp^\ast_{L,p})(  \beta^b_S (T)) \big)\Bigr)\\ = \,&e_\theta \Bigl(\big [ ({\bigwedge}^g_{R_p}\exp^\ast_{L,p}) ( \beta^b_S (T)), \, c \big ] \cdot \frac{\Omega (A)_{b, \omega}}{\Omega^p(A)_{c, \omega}}\Bigr)\\
=\,&\Bigl(\big [e_\theta( ({\bigwedge}^g_{R_p}\exp^\ast_{L,p}) ( \beta^b_S (T))), \, e_\theta(c) \big ]_\theta \cdot \frac{\Omega (A,\theta)_{b, \omega}}{\Omega^p(A,\theta)_{c, \omega}}\Bigr)\cdot e_\theta,
\end{align*}
where $\Omega (A,\theta)_{b, \omega}$ and $\Omega^p(A,\theta)_{c, \omega}$ denote the elements of $\CC_p^\times$ that are defined by the respective equalities $e_\theta\cdot \Omega (A)_{b, \omega} = \Omega (A,\theta)_{b, \omega}\cdot e_\theta$ and $e_\theta\cdot \Omega^p (A)_{b, \omega} = \Omega^p (A,\theta)_{b, \omega}\cdot e_\theta$. \medskip \\
Given the non-degeneracy of the pairing $[-,-]_\theta$ on $e_\theta$-components, the bijectivity of the map 
$e_\theta(\overline{\Q_p}\otimes_{\Q_p}{\bigwedge}^g_{R_p}\exp^\ast_{L,p})$ and the explicit definition of the element $\eta^b_S (T)$, the last displayed equality implies that the equality $e_\theta \cdot \beta^b_S (T) = e_\theta\cdot \eta^b_S (T)$ predicted by the $e_\theta$-component of Conjecture ${\rm SS}(M / K, R)$ is valid if and only if the equality (\ref{soule-stark}) is valid.
\end{proof}

\begin{remark}
If $A$ is an elliptic curve and $K$ an imaginary quadratic field (in this case necessarily of class number one), then (\ref{soule-stark}) is a consequence of the classical reciprocity law of Wiles \cite{Wiles} (see, for example, \cite[Prop.\@ 4.2]{flach-survey-2}).
\end{remark}

\subsubsection{The equivariant Tamagawa Number Conjecture for abelian varieties}\label{etnc AV}

We shall now combine Proposition \ref{soule-stark-ab-varieties} with the results of Theorems \ref{descent} and \ref{det-imc-result} to obtain concrete evidence in support of an important special case of the equivariant Tamagawa Number Conjecture. \medskip \\
To do this we fix data $L/K, G, A$ and $T$ as at the beginning of \S\ref{exp rec BL}. We note that the action of the absolute Galois group $G_K$ of $K$ on $(T_p A)_{L /K}$ gives rise to a character 
\begin{align} \label{character-1}
& \psi  
 \: G_K \longrightarrow  \cR_p^\times,
\end{align}
and we write $L_\infty$ for the abelian extension of $K$ that corresponds to the subgroup $\ker(\psi)$ of $G_K$. 

We note that this definition of $L_\infty$ is consistent with that given in \S \ref{soule-stark-section} since the Weil pairing implies $\psi$ coincides with the character $\chi_T$ constructed in Definition \ref{character-definition}. 
 We further note that the natural injective homomorphism 
\[
\mathcal{O}_p [G] \hookrightarrow \bigoplus_{\theta \in \widehat{G}} \bigO_p [\im \theta];
\quad a \mapsto ( \theta (a))_{\theta \in \widehat{G}},
\]
implies that $\Gal (L_\infty / K)\cong \im(\psi)$ is isomorphic to a subgroup of $\bigoplus_{\theta \in \widehat{G}} \bigO_p [\im \theta]^\times$ and hence has the form $\Z_p^d \times \Delta$ with $d$ a non-negative integer and $\Delta$ a finite abelian group. \\
We write $\Delta'$ for the maximal subgroup of $\Delta$ of order prime-to-$p$ and note that the subset $\Upsilon$ of $\widehat{\Delta'}$ comprising characters that satisfy Hypothesis \ref{bss-hypothesis} is stable under the natural action of $\Gal(\overline{\Q_p}/\Q_p)$ on $\widehat{\Delta'}$ and hence that the idempotent
\[ \varepsilon_{L,1} := \sum_{\chi \in \Upsilon}e_\chi\]
belongs to $\ZZ_p[\Delta'] \subset \ZZ_p\llbracket\Gal(L_\infty/K)\rrbracket$. 


We denote the subset of $\widehat{G}$ comprising all characters $\theta$ for which both $e_\theta (\overline{\Q_p} \otimes_{\ZZ} A (L))$ vanishes and $L_S(\psi, \theta^{-1},1)\not= 0$ by  $\widehat{G}_{0,A}$ and note that the associated idempotent 
\[ \varepsilon_{L,0} := \sum_{\theta\in\widehat{G}_{0,A}} e_\theta  \]
belongs to $\Q[G]\subset \Q_p[G]$. Hence, we obtain an idempotent of $K_p[G]$ by setting  
\[ \varepsilon_{L} := \psi(\varepsilon_{L,1})\cdot\varepsilon_{L,0}.\]

\begin{theorem}\label{final cor} Assume the data $L/K, A$ and $p$ are such that all of the following conditions are satisfied.  

\begin{enumerate}[label=(\roman*)]
\item $A (L)$ has no element of order $p$.
\item $\sha (A / L) [p^\infty]$ is finite.
\item The explicit reciprocity conjecture (\ref{soule-stark}) is valid for all characters $\theta$ in $\widehat{G}_{0,A}$. 
\item If $\psi(G_K)[p^\infty]$ is non-zero, then the $\ZZ_p$-module $\Cl (L_\infty)$ is finitely generated.
\end{enumerate}

Then, if the Rubin-Stark Conjecture holds for all finite abelian extensions of $K$, the equivariant Tamagawa Number Conjecture is valid for the pair $(h^1 (A / L) (1), 
\mathcal{O}_p[G]\varepsilon_{L})$. 
\end{theorem}

\begin{proof} The $\varepsilon_{L,0}$-component 
$M' := h^1 ( A / L) (1)\varepsilon_{L,0}$ of $h^1 ( A / L) (1)$ has a natural action of the $K$-algebra $R' := K[G]\varepsilon_{L,0}$. We define an $\mathcal{O}_p$-order $\mathcal{R}_p':= \mathcal{O}_p[G]\varepsilon_L$ in $(K_p\otimes_K R')\psi(\varepsilon_{L,1})$ and an $(\mathcal{R}_p'\times G_K)$-module $T' := \mathcal{R}_p'\otimes_{\mathcal{O}_p[G]}T$.  \medskip \\
 %
Then, with this notation, the basic strategy in this argument is to apply Theorem \ref{descent} with the data $\varepsilon$, $\mathcal{R}_p$, $R$,
 $T$, $M$ and $V$ taken to be $\varepsilon_{L,1}$, $\mathcal{R}_p'$, $R'$, $T'$, $M'$ and $V' := \QQ_p\otimes_{\Z_p}T'$ respectively and with the field $L_\infty$ as specified above. We must therefore verify that all of the necessary hypotheses are satisfied by this choice of data. \medskip \\
For each $n$ the kernel of the homomorphism $G_K \to {\rm Aut}_{\cR/p^n}(T/p^n) \cong (\cR/p^n)^\times$ induced by $\psi$ is contained in the kernel of the map $\chi_{T',n}$ from Definition \ref{character-definition} 
 and so Remark \ref{extended field remark} implies that this choice of $L_\infty$ is permissible in Theorem \ref{descent} for the above choice of data. \\
We also note that no finite place of $K$ can split completely in $L_\infty$. Indeed, the decomposition group in $G_K$ of any such place $v$ must be contained in $\ker(\psi)$ and hence acts trivially on $T \cong (T_p A)_{L/K}$ and this is not possible since $A (L_w)[p^\infty]$ is a finitely generated $\Z_p$-module for each place $w$ of $L$ above $v$. Since this argument also shows that the $\Z_p$-rank of $\Gal (L_\infty / K )$ is at least one, it follows that the conditions in the first sentence of Theorem \ref{descent} are satisfied in this case. \medskip \\
We claim next that Hypotheses \ref{hyp} and \ref{hyptorsionfree} are satisfied with $\mathcal{R}_p$ and $T$ taken to be $\mathcal{R}_p'$ and $T'$ respectively. \\
This is true for Hypotheses \ref{hyp} as $Y_K(T')$ is isomorphic  to $\mathcal{R}_p'\otimes_{\mathcal{O}_p}Y_K(T_pA)$ as an $\mathcal{R}_p'$-module and is true for Hypothesis \ref{hyptorsionfree}\,(i) since 
$H^0(K,(T')^*(1))$ is a submodule of the space $H^0(K,(V')^*(1)) = \varepsilon_L\bigl( H^0(L, (V_pA)^*(1))\bigr)$ and  $H^0(L, (V_pA)^*(1))$ vanishes. Concerning Hypothesis \ref{hyptorsionfree}\,(ii) we note first that the stated condition (i) implies $H^1(\mathcal{O}_{K, S},T^*(1))$ is $\ZZ_p$-torsion-free since its torsion submodule is isomorphic to 
\[ H^0(\mathcal{O}_{K, S},V^*(1)/T^*(1)) \cong 
 H^0(\mathcal{O}_{L, S},V_pA/T_pA) \cong A(L)[p^\infty].\]
This fact then combines with the results of Lemma \ref{compact lemma}\,(iii) (with $\cR_p$ taken to be $\mathcal{O}_p[G]$) and  Lemma \ref{compact lemma}\,(iv) (for the homomorphism $\mathcal{O}_p[G] \to \cR_p'$) to imply that $H^1(K,(T')^*(1))$ is isomorphic to the kernel of an endomorphism of a free $\mathcal{R}_p'$-module and so is torsion-free, as required.   \medskip \\
Finally, we note Hypothesis \ref{hypsoule} is satisfied with $\Lambda, \varepsilon$ and $\mathcal{R}_p$ taken to be 
$\Z_p\llbracket\Gal(L_\infty/K)\rrbracket$, $\varepsilon_{L,1}$ and $\mathcal{R}_p'$ respectively: the validity of Hypothesis \ref{hypsoule}\,(i)\,(a) in this case follows directly from the fact that $\varepsilon_L$ is the identity element of  
$\mathcal{R}_p'$, the validity of Hypothesis \ref{hypsoule}\,(i)\,(b) is clear and the validity of Hypothesis \ref{hypsoule}\,(i)\,(c) follows from the fact that the Teichm\"uller character does not belong to the set $\Upsilon$ of characters that is used to define $\varepsilon_{L,1}$.  \\
Now, since the validity of the Integrality Conjecture for all extensions $L_n/K$ follows directly from Remark \ref{RSremark} and the assumption that the Rubin-Stark Conjecture is valid for all finite abelian extensions of $K$, the above observations imply that the result of Theorem \ref{descent} implies the equivariant Tamagawa Number Conjecture is valid for the pair $(M',\mathcal{R}_p')$ provided that all of the conditions (a), (b) and (c) that occur in the latter result are satisfied in this case. \\
In addition, the condition in Theorem \ref{descent}\,(a) is satisfied since the stated condition (iv) and the assumed validity of the Rubin-Stark Conjecture combine with Theorem \ref{det-imc-result} to imply Conjecture \ref{IMCgeneral} is valid for $(L_\infty/K,S)$ with respect to the order $\mathfrak{A} = \Z_p\llbracket\Gal(L_\infty/K)\rrbracket\varepsilon_{L,0}$. \\
It is thus enough to note the conditions in Theorem \ref{descent}\,(b) and (c) are also satisfied since the definition of the idempotent $\varepsilon_{L,1}$ combines with the stated condition (ii) and the argument of Proposition \ref{soule-stark-ab-varieties} to imply both that $H^2(\mathcal{O}_{K,S},(V')^*(1)) \cong \varepsilon_L\cdot H^2(\mathcal{O}_{L,S},(V_pA)^*(1))$ vanishes and that the equality in Conjecture \ref{conj} is valid for the pair $(M',R')$. \medskip \\
This completes the proof of Theorem \ref{final cor}. 
\end{proof}

\printbibliography

\end{document}